%% file: main.tex
\documentclass[11pt]{article}

\usepackage[T1]{fontenc}
\usepackage[margin=1in]{geometry}
\usepackage{amsmath,amssymb,amsthm,mathtools,bm,mathrsfs}
\usepackage{enumitem}
\usepackage{booktabs}
\usepackage{microtype}
\usepackage{tikz}
\usetikzlibrary{arrows.meta,positioning,calc,fit,shapes.geometric}
\usepackage[hidelinks]{hyperref}

\newcommand{\E}{\mathbb E}
\newcommand{\KL}{D_{\mathrm{KL}}}
\newcommand{\cA}{\mathcal A}
\newcommand{\cB}{\mathcal B}
\newcommand{\cC}{\mathcal C}
\newcommand{\cG}{\mathcal G}
\newcommand{\cJ}{\mathcal J}
\newcommand{\cQ}{\mathcal Q}
\newcommand{\cR}{\mathcal R}
\newcommand{\cM}{\mathcal M}

\newcommand{\norm}[1]{\left\lVert #1\right\rVert}

\newtheorem{theorem}{Theorem}[section]
\newtheorem{proposition}[theorem]{Proposition}

\newtheorem{lemma}[theorem]{Lemma}
\theoremstyle{definition}
\newtheorem{definition}[theorem]{Definition}
\theoremstyle{remark}

\title{Decision-Sufficient Posterior Approximation}
\author{Sean Plummer}
\date{\today}

\begin{document}
\maketitle

\begin{abstract}
We investigate the consequences of requiring a posterior approximation to preserve a specified downstream decision problem.  A target posterior $P$ and loss determine a regret geometry on actions, a baseline approximation $Q_0$ determines the forward-Kullback--Leibler information required to induce action changes, and a restricted approximation family $\cQ$ determines which such changes are available.  Contracting KL divergence over Bayes-action fibers gives exact distances to decision adequacy and decision failure together with the least-informative posterior deformations that reach either side of the decision boundary.

In regular finite-dimensional problems, the target and baseline constructions have quadratic local limits: a target regret Hessian $G$ and a baseline information metric $J_I$.  Their generalized eigenproblem $Gv=\gamma J_Iv$ orders local decision directions by regret consequence per unit information cost and induces a tolerance-dependent effective dimension.  For restricted approximation families, the tangent image separates decision coverage from information efficiency: a family may miss consequential decision directions, or it may realize reachable directions only at excess Fisher cost.  The resulting framework provides decision-relative criteria for comparing and designing posterior approximation families.
\end{abstract}

\input{introduction}

\input{decision-adequacy}

\input{posterior-accessibility}

\input{local-geometry}

\input{effective-dimension}

\input{restricted-families}

\input{examples}

\input{discussion}

\section*{Acknowledgement of AI assistance}

The author used OpenAI's ChatGPT as an interactive research and writing assistant in developing this manuscript.  Assistance included exploring mathematical formulations, proposing and checking derivations, drafting and revising exposition and LaTeX, developing figures, and constructing numerical stress tests.  Outputs from the system were treated as suggestions requiring scrutiny, not as independent mathematical verification.  The author takes responsibility for the mathematical claims, references, and final contents of the paper.

\appendix

\input{appendix-regularity}

\input{appendix-contracted-information}

\input{appendix-local-proofs}

\input{appendix-restricted-families}

\bibliographystyle{plain}
\bibliography{references}

\end{document}

%% file: introduction.tex
\section{Introduction}
\label{sec:introduction}

We investigate the consequences of requiring a posterior approximation to preserve a specified downstream decision problem.  A Bayesian decision problem maps posterior uncertainty into an action through a loss.  Once that map is fixed, posterior discrepancies acquire different operational consequences: some move the chosen action through directions with substantial regret, while others have little effect at the decision scale of interest.  This suggests studying posterior approximation through the geometry induced by the decision rather than through a single distribution-wide discrepancy.

The construction separates three roles.  The target posterior $P$ and loss determine \emph{decision consequence}: which actions are acceptable and how costly nearby action errors are.  A baseline approximation $Q_0$ determines \emph{information accessibility}: the forward-KL information required to move the Bayes action in a prescribed direction.  A specified approximation family $\cQ$ inside an ambient candidate class of admissible measures determines \emph{family representability}: which of those decision movements are available inside the family and whether they can be realized at the ambient minimum Fisher cost.  The target posterior, baseline approximation, ambient candidate class, and restricted family play distinct roles throughout the paper; in particular, the target posterior need not belong to $\cQ$.

Using classical tools from Bayesian decision theory, I-projection, semiparametric efficiency, information geometry, and approximation theory, we develop a decision-relative theory of posterior approximation.  Contracting KL divergence over Bayes-action fibers produces exact distances to decision adequacy and decision failure.  In regular finite-dimensional problems, the target regret and baseline information constructions have quadratic local limits, represented by matrices $G$ and $J_I$.  Their generalized eigenproblem produces decision--information modes, ordered by regret consequence per unit information cost, and a tolerance-dependent effective dimension.  For restricted approximation families, the same geometry yields separate diagnostics for decision coverage and information efficiency.

These ingredients connect several established lines of work.  Loss-calibrated Bayesian approximation evaluates approximate inference through downstream utility or regret \cite{lacostejulien2011loss}.  Decision-theoretic divergences and information--risk relations describe how a loss selects task-relevant distinctions between distributions \cite{dawid2007geometry,williamson2024bridge}.  I-projection characterizes minimum relative-entropy perturbations under distributional constraints \cite{csiszar1975idivergence}, while Watson and Holmes study decision stability inside Kullback--Leibler neighborhoods of an approximate model \cite{watson2017approximate}.  Estimating-function and semiparametric theory provide the sensitivity, covariance, and tangent-space efficiency structure used in the local construction \cite{godambe1960optimum,godambeheyde1987quasi,vanderVaart1998}.  Goal-oriented Bayesian inverse problems use a related spectral reduction to identify directions that are both data-informed and relevant to a quantity of interest \cite{spantini2017goal}.  The contribution here is to organize these classical constructions around posterior approximation for a specified decision problem and to carry that organization through exact global contractions, local spectral reduction, and restricted-family diagnostics.

\subsection{The three levels of the problem}

Let $\theta$ denote the latent quantity, let $\cA$ be an action space, and let $L(\theta,a)$ be a loss.  Let $\mathfrak C_{Q_0}$ denote the ambient candidate class of decision-admissible probability measures used in the information contraction and anchored at the baseline $Q_0\in\mathfrak C_{Q_0}$, and let $\cQ\subseteq\mathfrak C_{Q_0}$ denote a specified restricted family when family geometry is under study.  The basic typing is
\[
Q\in\mathfrak C_{Q_0}
\stackrel{\Gamma}{\longmapsto}
a\in\cA.
\]
In the unique-action case, $a=\Gamma(Q)$; with ties, $\Gamma$ is understood set-valuedly through $\cB_0(Q)$.  The target posterior $P$ is conceptually separate from the candidate class $\mathfrak C_{Q_0}$: it determines the Bayes risk $R_P(a)=\E_P L(\theta,a)$ and the corresponding regret geometry, and it need not belong to $\cQ$ or to $\mathfrak C_{Q_0}$ unless a particular result explicitly requires that membership.  Starting from $Q_0$, we ask how much forward-KL information is required to make a specified action Bayes-optimal within the ambient candidate class $\mathfrak C_{Q_0}$.  This produces the ambient information cost on action space.  If $Q_0\in\cQ$, we can then compare that benchmark with the movements available inside the specified family.  Variational inference is recovered by choosing $\cQ$ to be a variational family.

In a regular finite-dimensional regime, these three levels reduce locally to a particularly simple picture.  The target side contributes a regret Hessian $G$, while the baseline side contributes an information metric $J_I$.  The generalized eigenproblem
\[
Gv_i=\gamma_iJ_Iv_i
\]
ranks action directions by regret curvature relative to the information required to induce them.  A restricted family then supplies a local tangent image in action space.  Its alignment with the leading decision--information modes determines decision coverage, while the Fisher cost of its lifts determines information efficiency.

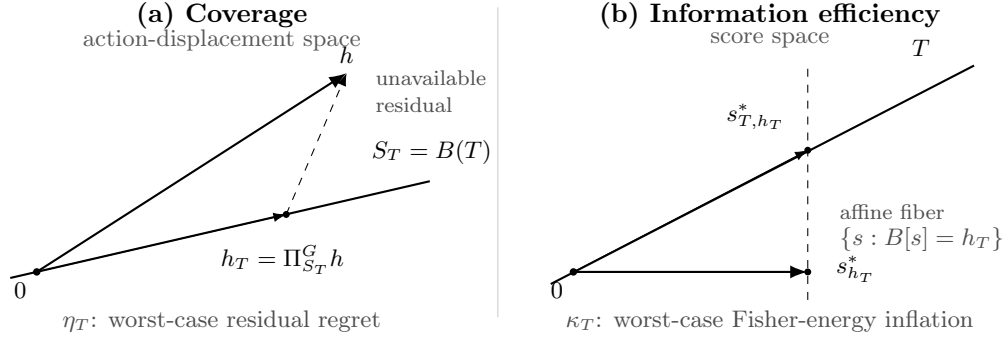
\begin{figure}[t]
\centering
\begin{tikzpicture}[
  >=Latex,
  every node/.style={font=\footnotesize}
]
  % (a) Coverage in action-displacement space
  \node[font=\small\bfseries] at (3.0,3.95) {(a) Coverage};
  \node[text=black!65] at (3.0,3.62) {action-displacement space};

  \coordinate (oA) at (0.55,0.55);
  \coordinate (hT) at (3.85,1.31);
  \coordinate (h) at (4.65,3.18);

  \draw[thick] (0.20,0.47) -- (5.75,1.75);
  \node[anchor=south west] at (4.85,1.85) {$S_T=B(T)$};

  \fill (oA) circle (1.4pt);
  \node[below left] at (oA) {$0$};

  \draw[->,thick] (oA) -- (h);
  \node[above] at (h) {$h$};

  \draw[->] (oA) -- (hT);
  \fill (hT) circle (1.3pt);
  \node[anchor=north] at (3.8,1.04) {$h_T=\Pi^{G}_{S_T}h$};

  \draw[->,dashed] (hT) -- (h);
  \node[align=left, anchor=west, text=black!70] at (4.90,2.94)
    {\scriptsize unavailable\\[-1pt]\scriptsize residual};

  \node[text=black!70] at (3.0,-0.10)
    {$\eta_T$: worst-case residual regret};

  % divider
  \draw[black!20] (6.65,-0.05) -- (6.65,4.05);

  % (b) Information efficiency in score space
  \node[font=\small\bfseries] at (10.25,3.95) {(b) Information efficiency};
  \node[text=black!65] at (10.25,3.62) {score space};

  \coordinate (oS) at (7.65,0.55);
  \coordinate (samb) at (10.75,0.55);
  \coordinate (sres) at (10.75,2.16);

  \draw[thick] (7.35,0.39) -- (12.95,3.28);
  \node[anchor=south] at (12.25,3.27) {$T$};

  \draw[dashed] (10.75,0.18) -- (10.75,3.30);
  \node[anchor=north west, align=left, text=black!70] at (11.04,1.58)
    {\scriptsize affine fiber\\[-1pt]$\{s:B[s]=h_T\}$};

  \fill (oS) circle (1.4pt);
  \node[below left] at (oS) {$0$};

  \draw[->,thick] (oS) -- (samb);
  \fill (samb) circle (1.3pt);
  \node[anchor=west] at (10.99,0.59) {$s_{h_T}^*$};

  \draw[->] (oS) -- (sres);
  \fill (sres) circle (1.3pt);
  \node[above left] at (10.58,2.30) {$s_{T,h_T}^*$};

  \node[text=black!70] at (10.25,-0.10)
    {$\kappa_T$: worst-case Fisher-energy inflation};
\end{tikzpicture}
\caption{Restricted-family decision geometry.  Left: $S_T=B(T)$ contains the reachable first-order action displacements; the $G$-projection residual $h-h_T$ is the local coverage error underlying $\eta_T$.  Right: for a reachable displacement $h_T$, the ambient minimum-Fisher-energy lift $s_{h_T}^*$ is the closest point to the origin on the affine fiber $\{s:B[s]=h_T\}$, while restriction to the linear tangent space $T$ may require the higher-energy lift $s_{T,h_T}^*$.  Taking the corresponding worst cases over information-bounded or reachable directions gives $\eta_T$ and $\kappa_T$.}
\label{fig:decision-relative-schematic}
\end{figure}

\subsection{Contributions and scope}

The paper develops the construction in three stages.  First, regret-acceptable action sets and the contracted information $I_{Q_0}(a)$ define exact global distances to decision adequacy and strict failure.  Second, in the regular local regime, the pair $(G,J_I)$ determines decision--information modes and a tolerance-dependent effective dimension.  Third, for a restricted approximation family, the tangent image identifies the decision directions represented by the family, while restricted Fisher energy measures the additional information cost of reaching those directions.

The goal-oriented linear--Gaussian reduction of Spantini et al. \cite{spantini2017goal} is recovered as a special case of the decision--information eigenproblem in Section~\ref{sec:effective-dimension}.  Here the two quadratic forms arise directly from downstream Bayes regret and baseline-relative information accessibility, so the construction applies beyond linear inverse problems and beyond prior-to-posterior reductions.

The local theory concerns regular finite-dimensional decision functionals with nondegenerate first-order response.  The global regret and information contractions remain well defined outside that regime; Section~\ref{sec:discussion} discusses boundary, discontinuous, and higher-order extensions.

\subsection{Organization}

Sections~\ref{sec:decision-adequacy} and~\ref{sec:posterior-accessibility} introduce the target and baseline constructions.  Sections~\ref{sec:local-geometry}--\ref{sec:restricted-families} develop the local geometry, oracle effective dimension, and restricted-family diagnostics.  Section~\ref{sec:examples} works through predictive scoring, mean-field Gaussian restrictions, and a small structured-Gaussian comparison.  Section~\ref{sec:discussion} discusses implications for variational-family design and extensions beyond the regular local regime.

%% file: decision-adequacy.tex
\section{Decision adequacy: what counts as good enough?}
\label{sec:decision-adequacy}

Let $P$ be the target posterior, let $\cA$ be the action space, and let $L(\theta,a)$ be the downstream loss.  Write
\[
R_P(a)=\E_P L(\theta,a),
\qquad
R_P^*=\inf_{b\in\cA}R_P(b),
\qquad
\Delta_P(a)=R_P(a)-R_P^*.
\]
The centered risk $\Delta_P$ records the decision consequence of an action error under the target posterior.

\subsection{Regret-acceptable actions}

\begin{definition}[Regret-acceptable action set]
For $\epsilon\ge0$, define
\[
\cB_\epsilon(P)=\{a\in\cA:\Delta_P(a)\le\epsilon\}.
\]
We call the nested family $\mathscr B(P)=\{\cB_\epsilon(P):\epsilon\ge0\}$ the \emph{Bayes-regret filtration}, or simply the \emph{regret filtration}.
\end{definition}

Here ``filtration'' means only a nested family of regret sublevel sets.  At zero tolerance, $\cB_0(P)=\arg\min_{a\in\cA}R_P(a)$ is the Bayes correspondence.  As $\epsilon$ grows, the set records which additional actions become acceptable for the target problem.

\begin{proposition}[Reconstruction of Bayes regret]
\label{prop:regret-reconstruction}
For every $a\in\cA$,
\[
\Delta_P(a)
=
\inf\{\epsilon\ge0:a\in\cB_\epsilon(P)\}.
\]
Hence $\mathscr B(P)$ and the centered risk $a\mapsto R_P(a)-R_P^*$ determine one another.
\end{proposition}

\begin{proof}
By definition, $a\in\cB_\epsilon(P)$ if and only if $\Delta_P(a)\le\epsilon$.  The admissible tolerances are therefore $[\Delta_P(a),\infty)$.
\end{proof}

The filtration repackages Bayes regret as nested sets of acceptable actions.  Its zero level retains exact optimality, while its positive levels retain the shape of the loss around the optimum.

\subsection{Exact decision equivalence versus positive-tolerance structure}

The zero level induces the equivalence relation $P\sim_0 Q$ if and only if $\cB_0(P)=\cB_0(Q)$.  When the Bayes action is unique, write it as $\Gamma(P)$, so $P\sim_0Q$ means $\Gamma(P)=\Gamma(Q)$.  Many posterior laws can therefore be exactly equivalent for the decision.

Exact Bayes equivalence leaves the local cost of action error unspecified.  In a regular problem with isolated interior Bayes action $a_P$ and twice differentiable risk,
\[
\Delta_P(a_P+h)=\tfrac12h^\top G_P h+o(\|h\|^2),
\qquad
G_P=\nabla_a^2R_P(a_P).
\]
The small-tolerance shape of $\cB_\epsilon(P)$ therefore records which nearby action directions accumulate regret most quickly.

For some losses, the whole filtration factors through the Bayes action.  More precisely, suppose there is a nonnegative action-space discrepancy $d$ such that, for every posterior in a class,
\[
\Delta_P(a)=d\!\left(\Gamma(P),a\right).
\]
Then all regret levels are determined by $\Gamma(P)$.  Squared-error estimation has this property.  Quantile loss generally does not: the exact quantile determines the zero-level action, while positive regret also depends on how probability mass is distributed around it.

\subsection{Decision adequacy}

\begin{definition}[Decision adequacy]
An approximate posterior $Q$ is \emph{$\epsilon$-adequate} for target posterior $P$ if
\[
\cB_0(Q)\cap\cB_\epsilon(P)\neq\varnothing.
\]
If the Bayes action under $Q$ is unique, this reduces to $\Delta_P(\Gamma(Q))\le\epsilon$.  In that case we call $Q$ \emph{strictly $\epsilon$-adequate} when $\Delta_P(\Gamma(Q))<\epsilon$.
\end{definition}

With nonunique Bayes actions, adequacy is existential: at least one Bayes action under $Q$ must be target-acceptable.  A uniform version would require $\cB_0(Q)\subseteq\cB_\epsilon(P)$, so that every Bayes action under $Q$ is acceptable.  The regular local theory assumes unique Bayes actions, where these notions coincide.

Decision adequacy is deliberately indexed by the downstream task.  At positive tolerance, two posterior laws can induce different Bayes actions while both remaining adequate for the same target decision problem.  Applications that require the full regret landscape, simultaneous preservation of several losses, or distribution-wide uncertainty statements can impose correspondingly richer adequacy requirements.

\subsection{Three simple cases}

The mean, quantile, and classification examples show three qualitatively different shapes for the regret filtration.

\paragraph{Squared-error mean.}
Let $g(\theta)$ be scalar and $L(\theta,a)=(a-g(\theta))^2$.  With $m_P=\E_Pg(\theta)$,
\[
\Delta_P(a)=(a-m_P)^2,
\qquad
\cB_\epsilon(P)=[m_P-\sqrt\epsilon,m_P+\sqrt\epsilon].
\]
Every regret level is determined by the posterior mean alone.

\paragraph{Quantile loss.}
For pinball loss at level $\tau$, assume $\E_P|Y|<\infty$ and let $q_\tau$ be the unique target $\tau$-quantile.  Then
\[
\Delta_P(a)=\int_{q_\tau}^{a}\{F_P(t)-\tau\}\,dt.
\]
If $F_P$ has density $f_P$ continuous and positive at $q_\tau$, then
\[
\Delta_P(q_\tau+u)=\tfrac12f_P(q_\tau)u^2+o(u^2).
\]
Two posteriors can therefore have the same exact quantile but different local regret for nearby actions.

\paragraph{Binary classification.}
Under zero-one loss with $\eta_P=P(Y=1)\neq1/2$, let
\[
a_P=\mathbf 1\{\eta_P>1/2\}
\]
denote the unique Bayes class.  The wrong class has regret $\rho_P=|2\eta_P-1|$, so
\[
\cB_\epsilon(P)
=
\begin{cases}
\{a_P\},&0\le\epsilon<\rho_P,\\
\{0,1\},&\epsilon\ge\rho_P.
\end{cases}
\]
Here the acceptable set jumps rather than expanding through a quadratic neighborhood.  This is the simplest example of behavior outside the regular local theory developed later.

%% file: posterior-accessibility.tex
\section{Posterior Accessibility: Distance to Adequacy and Failure}
\label{sec:posterior-accessibility}

The regret filtration specifies which actions are acceptable for the target.  Posterior accessibility adds a baseline-relative cost: how much forward-KL information must be introduced before a given action becomes Bayes-optimal?

\subsection{Action fibers and baseline-relative accessibility}

For $a\in\cA$, define the Bayes-action fiber within the ambient candidate class $\mathfrak C_{Q_0}$ by
\[
\cC_a=\{Q\in\mathfrak C_{Q_0}:a\in\cB_0(Q)\}.
\]
When the Bayes action is unique, with Bayes-action map $\Gamma$, the sets $\cC_a=\{Q\in\mathfrak C_{Q_0}:\Gamma(Q)=a\}$ are the fibers of $\Gamma$ and partition the ambient candidate class.  When ties are possible, the fibers can overlap because one posterior may have several Bayes-optimal actions.

Fix a baseline approximation $Q_0$.  Define
\begin{equation}
\label{eq:contracted-information-v2}
I_{Q_0}(a)
=
\inf_{Q\in\cC_a}\KL(Q\|Q_0).
\end{equation}
We call $I_{Q_0}(a)$ the \emph{contracted information} of $a$ relative to $Q_0$.  It is the least forward-KL change needed to move from the baseline to any posterior that makes $a$ Bayes-optimal.

Equation~\eqref{eq:contracted-information-v2} is an I-projection onto the action fiber \cite{csiszar1975idivergence}.  It is also the smallest KL-neighborhood radius at which $a$ becomes Bayesian-admissible in the sense of Watson and Holmes \cite{watson2017approximate}.  Thus $I_{Q_0}$ equips action space with a baseline-relative accessibility cost.

Unless stated otherwise, the infimum ranges over the ambient candidate class $\mathfrak C_{Q_0}$ rather than a specified restricted family.  This ambient problem is an information-theoretic benchmark: it identifies the cheapest candidate deformation capable of producing the action, against which restricted families can later be compared.  The same construction applies to any baseline approximation $Q_0$; variational restrictions enter only when we replace $\mathfrak C_{Q_0}$ by a specified variational family in Section~\ref{sec:restricted-families}.

The orientation $\KL(Q\|Q_0)$ is intentional.  It measures forward-KL information required to move away from the baseline and differs from the usual variational fidelity objective $\KL(Q\|P)$.  If the infimum is attained, write
\[
Q_a^*\in\arg\min_{Q\in\cC_a}\KL(Q\|Q_0).
\]
The value $I_{Q_0}(a)$ gives the information cost of the action displacement, while $Q_a^*$ gives the least-informative posterior deformation that achieves it.

\subsection{Distance to adequacy}

The posterior-level set of $\epsilon$-adequate distributions is
\[
\cG_\epsilon(P)=\{Q\in\mathfrak C_{Q_0}:\cB_0(Q)\cap\cB_\epsilon(P)\neq\varnothing\}.
\]
Define the \emph{decision repair radius}
\[
\cR_{P,Q_0}(\epsilon)
=
\inf_{Q\in\cG_\epsilon(P)}\KL(Q\|Q_0).
\]

\begin{proposition}[Repair as contraction onto the acceptable set]
\label{prop:repair-contraction-v2}
For every $\epsilon\ge0$,
\[
\cR_{P,Q_0}(\epsilon)
=
\inf_{a\in\cB_\epsilon(P)}I_{Q_0}(a).
\]
\end{proposition}

\begin{proof}
Since $\cG_\epsilon(P)=\bigcup_{a\in\cB_\epsilon(P)}\cC_a$, partial minimization over the union gives the result.
\end{proof}

If $Q_0$ is already $\epsilon$-adequate, then $\cR_{P,Q_0}(\epsilon)=0$.  Otherwise the minimizer identifies the cheapest acceptable action to reach, and its I-projection gives the corresponding posterior repair.  As the regret tolerance increases, the acceptable set grows and the repair radius can only decrease.

\subsection{Distance to failure and the common decision boundary}

Suppose the baseline is strictly $\epsilon$-adequate and the Bayes action is unique, with $a_0=\Gamma(Q_0)$ and $\Delta_P(a_0)<\epsilon$.  Strict failure means $\Delta_P(a)>\epsilon$.  Define the \emph{decision robustness radius} or \emph{failure radius} by
\[
\cM_{P,Q_0}(\epsilon)
=
\inf_{a:\Delta_P(a)>\epsilon}I_{Q_0}(a).
\]
Equivalently,
\[
\cM_{P,Q_0}(\epsilon)
=
\inf_{Q\in\mathfrak C_{Q_0}:\Delta_P(\Gamma(Q))>\epsilon}\KL(Q\|Q_0)
\]
when the Bayes action is unique.  By construction,
\[
\KL(Q\|Q_0)<\cM_{P,Q_0}(\epsilon)
\quad\Longrightarrow\quad
\Delta_P(\Gamma(Q))\le\epsilon.
\]
Thus $\cM$ is an information-space margin against decision failure.

Repair and robustness are contractions onto opposite sides of the same target decision boundary.  When $Q_0$ is inadequate, $\cR$ is the information needed to enter the acceptable region.  When $Q_0$ is strictly adequate, $\cM$ is the information needed to leave it.  The target supplies the boundary through $\cB_\epsilon(P)$, while the baseline supplies the cost of reaching different parts of that boundary through $I_{Q_0}(a)$.

The strict inequality in the definition of failure matters in discrete problems.  In regular continuous settings, when the boundary can be approached from the failure side without a jump in information cost, the strict-event infimum agrees with the infimum over the closure $\Delta_P(a)\ge\epsilon$.  The local quadratic theory uses this closed boundary problem.  With nonunique Bayes actions, the same global distinction can be formulated directly through $\cG_\epsilon(P)$ and its complement.

\subsection{Large-deviation interpretation}
\label{subsec:accessibility-large-deviations}

The forward-KL contraction also has a large-deviation interpretation.  Define the measure-level information
\[
\cJ_{Q_0}(Q):=\KL(Q\|Q_0).
\]
Then $I_{Q_0}(a)$ contracts $\cJ_{Q_0}$ onto the candidate-class Bayes-action fiber $\cC_a$, and $\cR$ and $\cM$ further contract the resulting action-level cost.

For the i.i.d. interpretation, let the $n$-sample law be
\[
Q_0^{(n)}:=Q_0^{\otimes n},
\qquad
(\theta_1,\ldots,\theta_n)\sim Q_0^{(n)},
\qquad
\widehat Q_n=\frac1n\sum_{i=1}^n\delta_{\theta_i}.
\]
All probabilities in this subsection are taken under $Q_0^{(n)}$.  Sanov's theorem gives $\cJ_{Q_0}$ as the empirical-measure rate function on its full probability-measure state space \cite{dembozeitouni1998}.  In settings where Bayes actions are well defined for empirical measures, let
\[
\widetilde{\cG}_\epsilon(P)
=
\bigl\{\mu\in\mathscr P(\Theta):
\cB_0(\mu)\cap\cB_\epsilon(P)\ne\varnothing\bigr\}
\]
be a measurable decision event in that state space.  This event differs from the candidate-class-restricted set $\cG_\epsilon(P)$ defined above.  When $\widetilde{\cG}_\epsilon(P)$ is a $\cJ_{Q_0}$-continuity set, Sanov yields
\[
-\lim_{n\to\infty}\frac1n
\log\Pr\{\widehat Q_n\in\widetilde{\cG}_\epsilon(P)\}
=
\inf_{\mu\in\widetilde{\cG}_\epsilon(P)}\cJ_{Q_0}(\mu).
\]
The exponent equals $\cR_{P,Q_0}(\epsilon)$ only if restricting this rate infimum to $\mathfrak C_{Q_0}$ does not change its value.  An analogous statement applies to strict failure.  Appendix~\ref{app:regularity-sanov} records these qualifications; without event continuity there are only the usual large-deviation upper and lower bounds.  For dependent samples, the corresponding rate function need not be $\cJ_{Q_0}$ \cite{dembozeitouni1998}.

\subsection{The running examples revisited}

The three examples from Section~\ref{sec:decision-adequacy} now acquire baseline-dependent information costs.

\paragraph{Gaussian mean.}
Let $Q_0=N(m_0,\sigma_0^2)$ and retain squared-error loss for a target with mean $m_P$.  The fixed-action contraction is
\[
I_{Q_0}(a)=\frac{(a-m_0)^2}{2\sigma_0^2}.
\]
Since the target accepts $[m_P-\sqrt\epsilon,m_P+\sqrt\epsilon]$,
\[
\cR_{P,Q_0}(\epsilon)
=
\frac{\bigl(|m_0-m_P|-\sqrt\epsilon\bigr)_+^2}{2\sigma_0^2}.
\]
If $|m_0-m_P|<\sqrt\epsilon$, continuity gives
\[
\cM_{P,Q_0}(\epsilon)
=
\frac{\bigl(\sqrt\epsilon-|m_0-m_P|\bigr)^2}{2\sigma_0^2}.
\]
The target determines the acceptable interval; the baseline variance determines how expensive its boundary is to reach.

\paragraph{Quantiles.}
For pinball loss at level $\tau\in(0,1)$, suppose $Q_0$ has continuous CDF $F_0$ and $\E_{Q_0}|Y|<\infty$.  The general Bayes condition is $F_Q(a^-)\le\tau\le F_Q(a)$.  Every finite-KL competitor satisfies $Q\ll Q_0$ and is therefore atomless, so on the finite-information fiber this reduces to $F_Q(a)=\tau$.  Consequently,
\begin{equation}
\label{eq:quantile-contracted-info}
I_{Q_0}(a)
=
\operatorname{kl}\!\left(\tau\,\middle\|\,F_0(a)\right),
\end{equation}
where
\[
\operatorname{kl}(p\|q)
=p\log\frac pq+(1-p)\log\frac{1-p}{1-q}
\]
denotes Bernoulli KL, with the usual extended-value convention at $q\in\{0,1\}$.  Appendix~\ref{app:contracted-info-v2} gives the two-set KL decomposition and the minimizing tilt.  The target determines which quantiles are acceptable, while $F_0$ determines how much information it costs to move to either boundary.

\paragraph{Binary classification.}
Suppose $\eta_P>1/2$ under zero-one loss and $0\le\epsilon<2\eta_P-1$.  Restricting to Bernoulli predictive laws with baseline probability $\eta_0$, the distance to the decision boundary is $\operatorname{kl}(1/2\,\|\,\eta_0)$.  It is a repair radius when $\eta_0<1/2$ and a strict-failure radius when $\eta_0>1/2$.  If $\epsilon\ge2\eta_P-1$, both actions are acceptable and strict failure is impossible, so $\cM=+\infty$.  This discrete example motivates retaining strict failure in the global definition even though the regular local problem uses the closed boundary.

%% file: local-geometry.tex
\section{Local Decision--Information Geometry}
\label{sec:local-geometry}

Near a regular decision, the global target and baseline constructions admit quadratic local descriptions.  The target contributes a regret curvature $G$ and the baseline contributes an information metric $J_I$.  Their interaction determines the local decision geometry.

Throughout this section the effective action space is an open subset of $\mathbb R^d$, the relevant Bayes actions are unique, and all local matrices are understood on effective action coordinates.  Appendix~\ref{app:regularity-v2} gives the precise regularity conditions.

\subsection{Target-side geometry: the local cost of decision error}

Let $P$ be a fixed target posterior with unique Bayes action $a_P$.  Under (T1)--(T4) of Appendix~\ref{app:regularity-v2}, $G:=\nabla_a^2R_P(a_P)\succ0$.  Because $\nabla_aR_P(a_P)=0$,
\begin{equation}
\label{eq:local-regret-expansion}
\Delta_P(a_P+h)
=
\frac12 h^\top Gh+o(\|h\|^2).
\end{equation}
The matrix $G$ measures the local regret consequence of action error: directions with larger curvature accumulate regret more quickly.

The same expansion controls the small-tolerance acceptable set.  The isolation assumption (T4) rules out remote nearly optimal actions and gives
\[
\frac{\cB_\epsilon(P)-a_P}{\sqrt{\epsilon}}
\longrightarrow
\left\{u:\frac12u^\top Gu\le1\right\}
\]
in Hausdorff distance as $\epsilon\downarrow0$.  Thus, near a regular isolated optimum, the regret filtration is locally a family of shrinking ellipsoids.  The matrix $G$ depends on the target posterior and the loss, so two targets with the same Bayes action can still assign different consequences to nearby action errors.

\subsection{Baseline-side geometry: the information cost of moving the action}

Fix a baseline approximation $Q_0$ with Bayes action $a_0$.  Under (I1)--(I5) of Appendix~\ref{app:regularity-v2}, actions near $a_0$ are identified by an equation
\[
\E_Q\psi_a=0.
\]
Here $\psi_a$ is an identification function for the Bayes action: among the finite-information competitors relevant to the contraction, the equation holds exactly when $a$ is Bayes-optimal.  For smooth strictly convex risks one can often take $\psi_a(\theta)=\nabla_aL(\theta,a)$; regular quantiles provide a useful nonsmooth example.

Define
\[
V=\operatorname{Cov}_{Q_0}(\psi_{a_0}),
\qquad
H=D_a[\E_{Q_0}\psi_a]_{a=a_0}.
\]
The local I-projection calculation in Appendix~\ref{app:contracted-info-v2} gives
\[
I_{Q_0}(a_0+h)
=
\frac12h^\top J_Ih+o(\|h\|^2),
\qquad
J_I=H^\top V^{-1}H.
\]
Thus $J_I$ is the local information cost of inducing an action displacement.  Algebraically it is the usual Godambe/optimal-estimating-function information \cite{godambe1960optimum,godambeheyde1987quasi}.

The same metric appears from pathwise differentiation.  Let $\mathcal H=L_0^2(Q_0)$ be the ambient score space and consider a regular path with $dQ_t/dQ_0=1+t s+o_{L^2}(t)$.  If $\Gamma(Q_t)=a_0+tB[s]+o(t)$, differentiating the identifying equation gives
\[
B[s]=-H^{-1}\E_{Q_0}[s\psi_{a_0}].
\]
Hence $BB^*=H^{-1}VH^{-\top}$ and
\[
J_I=(BB^*)^{-1}.
\]
This is the semiparametric covariance geometry in inverse form \cite{vanderVaart1998}: score energy is converted into a prescribed first-order displacement of the decision functional.

\paragraph{Geometric transport through the Bayes-action map.}
The derivative $B$ is the map through which the local posterior and decision geometries are related.  Posterior score space carries the Fisher inner product
\[
\langle s_1,s_2\rangle_{Q_0}=\E_{Q_0}[s_1s_2].
\]
Sending a score perturbation $s$ to the action displacement $h=B[s]$ and minimizing over all scores with the same image induces the quotient metric
\begin{equation}
\label{eq:quotient-information-metric}
h^\top J_Ih
=
\inf_{s:\,B[s]=h}
\norm{s}_{L^2(Q_0)}^2,
\qquad h\in E.
\end{equation}
Thus $J_I$ is the information metric induced on effective action space by the ambient Fisher geometry.  One can regard this as a pushforward only in the quotient sense, because $B$ is generally many-to-one.

Conversely, when the target regret curvature $G$ is expressed in the same local action coordinates, it pulls back through $B$ to the posterior score space:
\begin{equation}
\label{eq:pullback-regret-metric}
g_{\mathrm{dec}}(s_1,s_2)
=
B[s_1]^\top G B[s_2]
=
\langle s_1,B^*GB[s_2]\rangle_{Q_0}.
\end{equation}
The operator $B^*GB$ is positive semidefinite and vanishes on $\ker B$.  Posterior perturbations that do not move the Bayes action to first order are therefore null directions of the pulled-back decision geometry.  Passing to $E=\operatorname{ran}B$ quotients out exactly these invisible directions and yields the nondegenerate pair $(G,J_I)$ used below.  Appendix~\ref{app:regularity-pathwise} gives primitive identification-function conditions and a checkable smooth-loss sufficient condition for this transport.

\subsection{Minimum-information lifts}

Define the effective action space locally as
\[
E:=\operatorname{ran}B.
\]
Then $B:\mathcal H\to E$ is surjective by construction.  For a desired first-order displacement $h\in E$, consider
\[
\inf_{s\in\mathcal H:\,B[s]=h}
\frac12\|s\|_{L^2(Q_0)}^2.
\]
The unique minimum-energy, or horizontal, lift of $h$ is
\begin{equation}
\label{eq:minimum-information-score}
s_h^*=B^*J_Ih,
\end{equation}
with energy $\tfrac12h^\top J_Ih$.  The lift is orthogonal to $\ker B$, so any other score producing the same action displacement has larger Fisher energy.  Under the equality-fiber representation,
\[
s_h^*(\theta)=-\psi_{a_0}(\theta)^\top V^{-1}Hh.
\]
With the additional $L^2$ differentiability condition in Appendix~\ref{app:contracted-info-v2}, $s_h^*$ is also the first-order expansion of the exact exponential I-projection.

This distinction becomes important for restricted approximation families.  A family can realize $h$ while missing the ambient minimum-energy lift $s_h^*$, thereby paying excess Fisher cost, or it can fail to realize $h$ at first order.  Section~\ref{sec:restricted-families} separates these two phenomena.

\subsection{Repair and failure share the same local problem}

To combine target regret and baseline information, let $t\downarrow0$ index a local regime in which the target and baseline Bayes actions approach one another.  Let $P_t$ have Bayes action $a_t$ and suppose
\[
\frac{a_t-a_0}{t}\to\delta.
\]
The parameter $t$ is a localization parameter rather than a sample-size index.  The target--baseline action discrepancy is $O(t)$; regular regret and KL information are quadratic in local displacement, so their natural scale is $O(t^2)$.

Under (M1)--(M4) of Appendix~\ref{app:regularity-moving-target}, uniformly for $h$ in compact sets,
\begin{align*}
I_{Q_0}(a_0+th)
&=
\frac{t^2}{2}h^\top J_Ih+o(t^2),\\
\Delta_{P_t}(a_0+th)
&=
\frac{t^2}{2}(h-\delta)^\top G(h-\delta)+o(t^2).
\end{align*}
Write the local regret tolerance as $\epsilon_t=t^2c$ and define
\[
c_0=\frac12\delta^\top G\delta,
\]
the unrepaired local regret at the baseline.

\begin{theorem}[Local repair and robustness]
\label{thm:local-repair-robustness-v2}
Under (M1)--(M4) of Appendix~\ref{app:regularity-moving-target}, for each fixed $c\ge0$,
\[
\frac{\cR_{P_t,Q_0}(t^2c)}{t^2}
\longrightarrow
r(c;\delta)
:=
\min_h
\left\{
\frac12h^\top J_Ih:
\frac12(h-\delta)^\top G(h-\delta)\le c
\right\}.
\]
For each fixed $c>c_0$, so that the baseline is locally strictly adequate,
\[
\frac{\cM_{P_t,Q_0}(t^2c)}{t^2}
\longrightarrow
m(c;\delta)
:=
\min_h
\left\{
\frac12h^\top J_Ih:
\frac12(h-\delta)^\top G(h-\delta)\ge c
\right\}.
\]
When the limiting minimizer is unique, information-minimizing repaired actions converge after rescaling to the repair minimizer, and every asymptotically minimizing sequence of strict-failure actions converges after rescaling to the failure minimizer.
\end{theorem}

Locally, repair and robustness are therefore the same minimum-information problem on opposite sides of a target regret ellipsoid.  Repair finds the cheapest point inside the acceptable region; robustness finds the cheapest point on or beyond its boundary.  The global failure event remains strict, but in the regular regime the information cost is continuous at the boundary, so the limiting problem uses the closed constraint.

If $c\ge c_0$, the origin is already acceptable and $r(c;\delta)=0$.  At exact decision matching, $r(0;\delta)=\tfrac12\delta^\top J_I\delta$.  For $c>c_0$, the robustness minimizer lies on the boundary of the acceptable ellipsoid.

\subsection{Generalized decision--information modes}

Because $J_I\succ0$, information whitening transforms the pair $(G,J_I)$ into the symmetric positive-definite matrix
\[
A=J_I^{-1/2}GJ_I^{-1/2}.
\]
Its orthonormal eigenbasis induces $J_I$-orthonormal generalized eigenvectors $v_i$ satisfying
\[
Gv_i=\gamma_iJ_Iv_i,
\qquad
\gamma_1\ge\cdots\ge\gamma_d>0,
\qquad
v_i^\top J_Iv_j=\delta_{ij}.
\]
The Rayleigh quotient
\[
\frac{v^\top Gv}{v^\top J_Iv}
=
\frac{\text{local target-regret curvature in direction }v}
{\text{local posterior-information cost in direction }v}
\]
shows that $\gamma_i$ measures regret consequence per unit information.  Large generalized eigenvalues identify directions that become consequential after comparatively little posterior movement.

Write $\delta=\sum_i d_iv_i$ and $h=\sum_i x_iv_i$.  Then
\[
\frac12h^\top J_Ih
=
\frac12\sum_i x_i^2,
\qquad
\frac12(h-\delta)^\top G(h-\delta)
=
\frac12\sum_i\gamma_i(x_i-d_i)^2.
\]
The generalized modes therefore decouple the local geometry into scalar channels: information cost is Euclidean in the $x_i$, while the target regret contribution of channel $i$ is weighted by $\gamma_i$.  This is the basis for ranking directions and for the effective-dimension calculation in Section~\ref{sec:effective-dimension}.

For repair with $0<c<c_0$, the active-constraint solution has
\[
x_i^{\mathrm{rep}}
=
\frac{\lambda\gamma_i}{1+\lambda\gamma_i}d_i
\]
for a unique $\lambda>0$ chosen to satisfy the regret boundary.  The corresponding failure KKT system and its trust-region hard case are given in Appendix~\ref{app:local-proofs-v2}.  The centered case $\delta=0$ makes the leading mode especially transparent:
\[
m(c;0)
=
\min_h
\left\{
\frac12h^\top J_Ih:
\frac12h^\top Gh\ge c
\right\}
=
\frac{c}{\gamma_1}.
\]
Any leading generalized mode, scaled as $h^*=\sqrt{2c/\gamma_1}\,v_1$, is a cheapest local failure direction.

\subsection{What the pair \texorpdfstring{$(G,J_I)$}{(G, JI)} tells us}

The pair separates two questions that a single posterior discrepancy cannot.  The target metric $G$ measures the consequence of action error; the baseline metric $J_I$ measures how much information is required to create that error.  Their generalized eigenvectors identify the directions that are most consequential relative to their accessibility.

Section~\ref{sec:effective-dimension} next asks an oracle representation question: if only a few decision directions can be retained, which subspace gives the best possible guarantee?  Section~\ref{sec:restricted-families} then asks whether a chosen approximation family realizes that oracle geometry and at what Fisher cost.

%% file: effective-dimension.tex
\section{How Many Decision Directions Matter?}
\label{sec:effective-dimension}

We now pose an oracle representation problem in action space.  Before choosing a particular variational family, suppose we may retain any $r$-dimensional subspace of local decision displacements.  Which subspace gives the smallest worst-case regret over perturbations of a specified information size?  Section~\ref{sec:restricted-families} will then ask whether a concrete approximation family realizes that oracle subspace.

\subsection{Worst-case representation over an information ball}

Work in the regular effective action space $E\simeq\mathbb R^d$ with $J_I\succ0$ and $G\succ0$.  For an information budget $\rho>0$, define
\[
\mathcal H_\rho
=
\left\{h\in E:\frac12h^\top J_Ih\le\rho\right\}.
\]
This is the set of local decision displacements that can be induced from the baseline at information cost at most $\rho$ to second order.

Now retain only an $r$-dimensional subspace $S\subset E$.  If the true displacement is $h$, the best approximation inside $S$ leaves residual regret
\[
\inf_{s\in S}\frac12(h-s)^\top G(h-s).
\]
Define
\[
\mathcal E_\rho(S)
=
\sup_{h\in\mathcal H_\rho}
\inf_{s\in S}
\frac12(h-s)^\top G(h-s),
\qquad
\mathcal E_r^*(\rho)
=
\inf_{\dim S=r}\mathcal E_\rho(S).
\]
This is a finite-dimensional Kolmogorov-width problem for an ellipsoid \cite{pinkus1985nwidths}.

Let $Gv_i=\gamma_iJ_Iv_i$ with $\gamma_1\ge\gamma_2\ge\cdots\ge\gamma_d>0$ and normalization $v_i^\top J_Iv_j=\delta_{ij}$.  Equivalently, if $A=J_I^{-1/2}GJ_I^{-1/2}$, then the $\gamma_i$ are the eigenvalues of $A$ and $u_i=J_I^{1/2}v_i$ are an orthonormal eigenbasis.

\begin{theorem}[Optimal information-bounded decision representation]
\label{thm:effective-dimension}
With the convention $\gamma_{d+1}=0$, for every $r\in\{0,\ldots,d\}$,
\[
\mathcal E_r^*(\rho)=\rho\gamma_{r+1}.
\]
An optimal retained subspace is $S_r^*=\operatorname{span}\{v_1,\ldots,v_r\}$, with $S_d^*=E$.  The smallest rank that guarantees omitted regret at most $\epsilon$ is
\[
 d_{\mathrm{eff}}(\epsilon;\rho)
 =
 \min\{r:\rho\gamma_{r+1}\le\epsilon\}
 =
 \#\{i:\rho\gamma_i>\epsilon\}.
\]
\end{theorem}

\begin{proof}
Whiten action space by the information metric.  Let $x=J_I^{1/2}h$ and $L=J_I^{1/2}S$.  Then $\mathcal H_\rho$ maps to $\{x:\|x\|^2\le2\rho\}$ and
\[
\inf_{s\in S}\frac12(h-s)^\top G(h-s)
=
\inf_{y\in L}\frac12(x-y)^\top A(x-y).
\]
Because $A\succ0$, the map $z=A^{1/2}x$ is invertible.  The information ball is mapped to the ellipsoid
\[
\mathcal K_\rho
=
\left\{A^{1/2}x:\|x\|^2\le2\rho\right\},
\]
whose Euclidean principal semiaxes are $\sqrt{2\rho\gamma_1},\ldots,\sqrt{2\rho\gamma_d}$.  As $L$ ranges over $r$-dimensional subspaces, $A^{1/2}L$ ranges over all $r$-dimensional Euclidean subspaces.  For $r<d$, the Kolmogorov $r$-width of this ellipsoid is its $(r+1)$st semiaxis.  Squaring and multiplying by $1/2$ gives $\mathcal E_r^*(\rho)=\rho\gamma_{r+1}$, and the optimal subspace maps back to $\operatorname{span}\{v_1,\ldots,v_r\}$.  The case $r=d$ gives zero residual error.  The formula for $d_{\mathrm{eff}}$ follows.
\end{proof}

The retention rule is immediate: mode $v_i$ must be kept for a uniform $(\epsilon,\rho)$ guarantee exactly when $\rho\gamma_i>\epsilon$.  The resulting effective dimension is therefore a profile indexed by both the allowed regret and the information scale.  It is a property of the local decision problem, rather than of posterior dimension or parameter dimension alone.

\subsection{Tolerance and information scale}

The map $\epsilon\mapsto d_{\mathrm{eff}}(\epsilon;\rho)$ is a step function with breakpoints $\epsilon=\rho\gamma_i$.  Increasing the allowed regret can only reduce the effective dimension, while increasing the perturbation radius can only increase it.  At zero tolerance every effective local decision direction must be retained, whereas at positive tolerance some genuine directions of posterior-induced decision variation can be omitted.

The construction is invariant to invertible linear reparameterization of the action: $G$ and $J_I$ transform by congruence and the generalized eigenvalues do not change.  The profile therefore belongs to the decision geometry, not to a chosen coordinate system.

\subsection{Structural, realized, and ensemble importance}
\label{subsec:ensemble-dimension}

The width calculation is deliberately worst-case.  It treats every displacement inside $\mathcal H_\rho$ as possible.  A realized target mismatch can emphasize a much smaller subset of the structural modes.  If
\[
\delta=\sum_{i=1}^d c_i v_i,
\]
then
\[
\frac12\delta^\top G\delta
=
\frac12\sum_{i=1}^d\gamma_i c_i^2.
\]
The quantity $\tfrac12\gamma_i c_i^2$ is the realized regret contribution of mode $i$.  A mode with a large structural eigenvalue can contribute little to one realized target mismatch if its coefficient $c_i$ is small.

For a single fixed $\delta$, arbitrary subspace selection makes a dimension notion nearly trivial because $\operatorname{span}\{\delta\}$ represents it exactly.  A nontrivial representation problem appears when one subspace must work over a class or distribution of plausible mismatches.  Let $\Delta$ be a random local target displacement with finite second moment and define
\[
\overline{\mathcal E}(S)
=
\E\left[
\inf_{s\in S}
\frac12(\Delta-s)^\top G(\Delta-s)
\right],
\qquad
\overline{\mathcal E}_r^*
=
\inf_{\dim S=r}\overline{\mathcal E}(S).
\]
Let $M_\Delta=\E[\Delta\Delta^\top]$, define $K=G^{1/2}M_\Delta G^{1/2}$, and write its eigenvalues as $\zeta_1\ge\cdots\ge\zeta_d\ge0$.

\begin{proposition}[Target-ensemble representation]
\label{prop:ensemble-dimension}
For every $r\in\{0,\ldots,d\}$,
\[
\overline{\mathcal E}_r^*
=
\frac12\sum_{i>r}\zeta_i.
\]
Consequently, an average-case effective dimension at tolerance $\epsilon$ is
\[
 d_{\mathrm{ens}}(\epsilon)
=
\min\left\{
 r:
 \frac12\sum_{i>r}\zeta_i\le\epsilon
\right\}.
\]
\end{proposition}

\begin{proof}
Let $Y=G^{1/2}\Delta$ and $L=G^{1/2}S$.  Then $\overline{\mathcal E}(S)=\tfrac12\E[\operatorname{dist}^2(Y,L)]$.  The rank-$r$ minimizer is the principal-component, or proper-orthogonal-decomposition, subspace of the second-moment matrix $K=\E[YY^\top]$, and the minimum residual mean-square error is the sum of the omitted eigenvalues \cite{jolliffe2002pca}.
\end{proof}

In information-whitened coordinates, let $Z=J_I^{1/2}\Delta$, $\Omega=\E[ZZ^\top]$, and $A=J_I^{-1/2}GJ_I^{-1/2}$.  The nonzero eigenvalues governing the ensemble problem are those of $A^{1/2}\Omega A^{1/2}$.  Average-case complexity therefore depends on how costly an error is through $G$, how easy it is to induce through $J_I$, and how often the target ensemble points in that direction through $\Omega$.

If $\Omega$ is diagonal in the generalized decision--information basis and $\Delta=\sum_i C_i v_i$ with $\E[C_iC_j]=0$ for $i\ne j$, then the ensemble eigenvalue multiset is $\{\gamma_i\E[C_i^2]:i=1,\ldots,d\}$, reordered from largest to smallest.  The average contribution attached to mode $i$ is $\tfrac12\gamma_i\E[C_i^2]$, and this ordering need not agree with the worst-case ordering of the $\gamma_i$.

If the target ensemble is supported inside $\mathcal H_\rho$, then $\overline{\mathcal E}_r^*\le\mathcal E_r^*(\rho)=\rho\gamma_{r+1}$ for every $r$, and hence $d_{\mathrm{ens}}(\epsilon)\le d_{\mathrm{eff}}(\epsilon;\rho)$.  A uniform guarantee is at least as demanding as an average one over a specified ensemble.

\subsection{Connection with goal-oriented linear--Gaussian reduction}
\label{subsec:spantini-special-case}

The goal-oriented linear--Gaussian theory of Spantini, Cui, Willcox, Tenorio, and Marzouk \cite{spantini2017goal} is recovered directly from $(G,J_I)$ in the corresponding Gaussian decision problem.  In this special case we take the prior QoI law as the baseline distribution from which information cost is measured.

Consider a linear--Gaussian inverse problem with a linear quantity of interest $Z$.  Let the baseline QoI law be the prior $Q_0=N(0,\Gamma_Z)$ and, after observing data $Y=y$, let the target QoI posterior be $P_y=N(\mu_y,\Gamma_{Z\mid Y})$.  Use posterior-precision-weighted squared loss
\[
L(z,a)=(a-z)^\top\Gamma_{Z\mid Y}^{-1}(a-z).
\]
The Bayes action is the posterior mean and
\[
\Delta_{P_y}(a)=(a-\mu_y)^\top\Gamma_{Z\mid Y}^{-1}(a-\mu_y),
\]
so $G=2\Gamma_{Z\mid Y}^{-1}$.  For the Gaussian baseline, the minimum KL cost of shifting the mean by $h$ is $I_{Q_0}(h)=\tfrac12h^\top\Gamma_Z^{-1}h$, hence $J_I=\Gamma_Z^{-1}$.

Spantini et al.'s goal-oriented QoI directions $\widehat q_i$ can be normalized so that $\widehat q_i^\top\Gamma_Z^{-1}\widehat q_j=\delta_{ij}$ and satisfy
\[
\Gamma_{Z\mid Y}\Gamma_Z^{-1}\widehat q_i=(1-\lambda_i)\widehat q_i.
\]
It follows that
\[
2\Gamma_{Z\mid Y}^{-1}\widehat q_i
=\frac{2}{1-\lambda_i}\Gamma_Z^{-1}\widehat q_i.
\]
Their QoI modes are therefore the generalized decision--information modes here, with $\gamma_i=2/(1-\lambda_i)$ and the same ordering.

The average-case criterion also matches.  Under the prior-predictive distribution of the data,
\[
\operatorname{Var}\{\mu_{Z\mid Y}(Y)\}=\Gamma_Z-\Gamma_{Z\mid Y}.
\]
Writing $\mu_{Z\mid Y}(Y)=\sum_i C_i(Y)\widehat q_i$ gives $\E[C_i(Y)^2]=\lambda_i$.  Hence the ensemble-average regret contribution of mode $i$ is $\lambda_i/(1-\lambda_i)$, exactly their posterior-mean Bayes-risk contribution.

The effective dimension is an oracle statement: it identifies the best rank-$r$ action-space representation that could exist.  A concrete approximation family can have the same rank while pointing in the wrong directions, or can reach the right directions at unnecessary Fisher cost.  Those realization questions are the subject of Section~\ref{sec:restricted-families}.

%% file: restricted-families.tex
\section{Restricted Approximation Families: Coverage and Information Efficiency}
\label{sec:restricted-families}

The oracle problem in Section~\ref{sec:effective-dimension} identifies which decision directions should be retained when any action-space subspace is available.  To study a concrete approximation method, fix a specified family of measures $\cQ\subseteq\mathfrak C_{Q_0}$, where $\mathfrak C_{Q_0}$ is the ambient candidate class introduced in Section~\ref{sec:introduction}.  The family $\cQ$ is the object whose restrictions we diagnose; the target posterior $P$ remains a separate object that determines regret and need not belong to $\cQ$.  The family restriction introduces two additional questions.  \emph{Coverage} asks which first-order decision displacements $\cQ$ can realize.  \emph{Information efficiency} asks whether those reachable displacements can be realized at the ambient minimum Fisher cost.  Variational inference is the special case in which $\cQ$ is chosen to be a variational family.

\subsection{Global family adequacy and restricted contraction}

For the specified family $\cQ\subseteq\mathfrak C_{Q_0}$, define its action image
\[
\operatorname{Act}(\cQ)=\bigcup_{Q\in\cQ}\cB_0(Q).
\]
The family is globally $\epsilon$-adequate for target $P$ if
\[
\operatorname{Act}(\cQ)\cap\cB_\epsilon(P)\neq\varnothing.
\]
Thus family adequacy is a statement about which target-relevant actions are represented by the chosen measure family.

Fix $Q_0\in\cQ$.  Define the restricted contracted information
\begin{equation}
\label{eq:restricted-contracted-information-v2}
I_{Q_0,\cQ}(a)
=
\inf\left\{
\KL(Q\|Q_0):
Q\in\cQ,\ a\in\cB_0(Q)
\right\},
\end{equation}
with value $+\infty$ when the restricted action fiber is empty.  The corresponding restricted repair radius is
\[
\cR_{P,Q_0,\cQ}(\epsilon)
=
\inf\left\{
\KL(Q\|Q_0):
Q\in\cQ,\
\cB_0(Q)\cap\cB_\epsilon(P)\neq\varnothing
\right\}.
\]

\begin{proposition}[Restricted repair contraction]
\label{prop:restricted-repair-contraction}
For every $\epsilon\ge0$,
\[
\cR_{P,Q_0,\cQ}(\epsilon)
=
\inf_{a\in\cB_\epsilon(P)}I_{Q_0,\cQ}(a)
=
\inf_{a\in\operatorname{Act}(\cQ)\cap\cB_\epsilon(P)}
I_{Q_0,\cQ}(a).
\]
Consequently
\[
\cR_{P,Q_0,\cQ}(\epsilon)\ge\cR_{P,Q_0}(\epsilon).
\]
\end{proposition}

\begin{proof}
Apply the same partial-minimization argument as in Proposition~\ref{prop:repair-contraction-v2} to the union of restricted action fibers.  The inequality follows because the restricted problem minimizes over $\cQ\subseteq\mathfrak C_{Q_0}$, a subset of the ambient candidate class.
\end{proof}

The ambient contraction over $\mathfrak C_{Q_0}$ is therefore the information-theoretic benchmark, while $I_{Q_0,\cQ}$ records the additional constraints imposed by the specified family.  Globally, an acceptable action can be unreachable inside $\cQ$, or it can be reachable only at greater KL cost.

\subsection{Local family geometry}

Return to the regular setting of Section~\ref{sec:local-geometry}.  Let $\mathcal H=L_0^2(Q_0)$ be the ambient score space and let
\[
E=\operatorname{ran}B
\]
be the effective action space, so $B:\mathcal H\to E$ is onto by construction and $J_I=(BB^*)^{-1}$.

For the local linear theory, let $\mathcal D_{\cQ}(Q_0)$ be the set of scores of regular paths that remain in $\cQ$ through $Q_0$.  Assume this set of feasible first-order scores is linear, and define the closed tangent space $T=\overline{\mathcal D_{\cQ}(Q_0)}^{\,L^2(Q_0)}\subset\mathcal H$, with orthogonal projector $\Pi_T$.  Thus directions in $T$ can be approximated by feasible family scores; taking an unrestricted linear span of a nonlinear tangent set would not justify the same coverage interpretation.  Define
\[
S_T:=B(T),
\qquad
C_T:=B\Pi_TB^*.
\]
The image $S_T$ contains the first-order decision displacements available to the family, while $C_T$ determines the minimum score energy of those displacements within $T$.

\begin{proposition}[Restricted tangent Fisher energy]
\label{prop:restricted-local-information-v2}
For $h\in E$,
\[
\inf_{s\in T:\,B[s]=h}
\frac12\norm{s}_{L^2(Q_0)}^2
=
\begin{cases}
\displaystyle \frac12h^\top C_T^\dagger h,&h\in S_T,\\[6pt]
+\infty,&h\notin S_T,
\end{cases}
\]
where $C_T^\dagger$ is the Moore--Penrose pseudoinverse.  In particular, the family has full first-order decision coverage, $B(T)=E$, if and only if $\operatorname{rank}(C_T)=\dim E$.
\end{proposition}

Appendix~\ref{app:restricted-family-details} gives the operator proof.  The statement concerns the linearized family: a direction in the closure $T$ is approximable by feasible first-order scores but need not be realized by a single finite move in the nonlinear family.  If feasible local directions form a cone or another nonlinear tangent set, their linear span can exaggerate available decision coverage and the present linear formulas require modification.

\subsection{Coverage, alignment, and effective dimension}

Full first-order decision coverage is stronger than necessary when positive regret is allowed.  A family may omit directions whose worst-case contribution lies below the decision tolerance.  Consider the ambient information ball
\[
\mathcal H_\rho=\left\{h:\frac12h^\top J_Ih\le\rho\right\}
\]
and define
\[
\mathcal E_\rho(T)
=
\sup_{h\in\mathcal H_\rho}
\inf_{s\in S_T}
\frac12(h-s)^\top G(h-s).
\]
If $S_T\neq\{0\}$ and $U$ is a full-column-rank matrix spanning $S_T$, let
\[
R_T=G-GU(U^\top GU)^{-1}U^\top G,
\]
and set $R_T=G$ when $S_T=\{0\}$.  Then
\[
\inf_{s\in S_T}\frac12(h-s)^\top G(h-s)
=
\frac12h^\top R_Th.
\]
Define the family coverage defect
\[
\eta_T
=
\lambda_{\max}(J_I^{-1/2}R_TJ_I^{-1/2}).
\]

\begin{proposition}[Restricted-family coverage certificate]
\label{prop:restricted-family-coverage-v2}
For every $\rho>0$,
\[
\mathcal E_\rho(T)=\rho\eta_T.
\]
Consequently the tangent image gives a uniform local regret guarantee $\epsilon$ over all ambient target displacements of information cost at most $\rho$ if and only if
\[
\rho\eta_T\le\epsilon.
\]
\end{proposition}

\begin{proof}
Whiten with $z=J_I^{1/2}h$.  The information constraint becomes $\norm z^2\le2\rho$, and the objective becomes $\tfrac12z^\top J_I^{-1/2}R_TJ_I^{-1/2}z$.  The maximum over the Euclidean ball is $\rho$ times the largest eigenvalue.
\end{proof}

The scalar $\eta_T$ measures alignment between the actual family image and the decision geometry.  It is zero exactly under full first-order decision coverage, while positive tolerance permits $\eta_T>0$ provided the resulting residual regret remains below budget.

Let $r_T=\dim S_T$ and let $\gamma_1\ge\cdots\ge\gamma_d>0$ be the generalized decision--information eigenvalues.  The oracle width result gives
\[
\eta_T\ge\gamma_{r_T+1},
\]
with $\gamma_{d+1}=0$.  Hence $r_T<d_{\mathrm{eff}}(\epsilon;\rho)$ rules out a uniform $(\epsilon,\rho)$ guarantee.  Rank alone is insufficient, however: two tangent images of the same dimension can have different values of $\eta_T$ because they point in different decision directions.  Equality with the oracle bound occurs when the image is aligned with the leading generalized modes.  Appendix~\ref{app:restricted-family-details} gives a two-dimensional rotated example.

\subsection{Information efficiency on reachable directions}

Coverage records which action directions are available.  A reachable displacement can still require more Fisher energy inside the family than in the ambient score space.  For $h\in S_T$, define
\[
I^{\mathrm{tan}}(h)=\frac12h^\top J_Ih,
\qquad
I_T^{\mathrm{tan}}(h)=\frac12h^\top C_T^\dagger h.
\]
Because the restricted problem minimizes over the smaller score space $T$,
\[
I_T^{\mathrm{tan}}(h)\ge I^{\mathrm{tan}}(h).
\]
When $S_T\neq\{0\}$, define the information-inflation factor
\[
\kappa_T
=
\sup_{h\in S_T\setminus\{0\}}
\frac{h^\top C_T^\dagger h}
{h^\top J_Ih}
\ge1.
\]
When $S_T=\{0\}$, set $\kappa_T=1$ by convention because there are no nonzero reachable directions on which to measure inefficiency.

For an ambient displacement $h$, the minimum-information lift is $s_h^*=B^*J_Ih$.  Therefore
\[
\kappa_T=1
\quad\Longleftrightarrow\quad
B^*J_I(S_T)\subseteq T.
\]
The family is Fisher-efficient on its reachable image exactly when it contains the ambient minimum-energy lift of every decision displacement it can realize.

The pair $(\eta_T,\kappa_T)$ therefore separates two family-specific properties at the baseline $Q_0$.  The quantity $\rho\eta_T$ is the worst residual regret caused by unavailable action directions over an ambient information ball of radius $\rho$.  The quantity $\kappa_T$ is the worst multiplicative Fisher-energy penalty along reachable directions.  Both depend on the chosen decision problem, the baseline, and the local geometry of the approximation family; neither is an intrinsic expressivity score for $\cQ$ in isolation.

%% file: examples.tex
\section{Examples: From Decision Geometry to Variational Restrictions}
\label{sec:examples}

The examples in this section isolate different parts of the construction.  Proper predictive scoring shows that one Bayes-action map can support several regret geometries.  Two centered Gaussian examples separate coverage from information efficiency for the same mean-field family.  A final structured-Gaussian comparison shows numerically that a family with better global KL fidelity can be worse for a specified downstream decision.

\subsection{Proper predictive scoring: the same action fibers, different regret geometry}
\label{subsec:proper-predictive-example-v2}

Let the outcome space be $\{1,\ldots,K\}$ and let $p(\theta)=(p_1(\theta),\ldots,p_K(\theta))\in\Delta^{K-1}$ be a predictive probability vector.  For posterior $Q$, define $m_Q=\E_Qp(\theta)$.  Under any strictly proper scoring rule the Bayes predictive action is $m_Q$, so all such scores share the action fibers $\cC_m=\{Q:\E_Qp=m\}$ \cite{dawid2007geometry,gneitingraftery2007strictly}.

What changes with the score is the regret assigned to action error.  Suppose the score is twice differentiable and strongly proper on the relevant simplex tangent space, so
\[
\Delta_P(m_P+u)
=
\frac12u^\top G_Su+o(\norm{u}^2),
\qquad G_S\succ0.
\]
By contrast, the local accessibility metric depends on the predictive map and the baseline posterior.  Writing
\[
\phi(\theta)=p(\theta)-m_0,
\qquad
\Sigma_p=\E_{Q_0}[\phi\phi^\top],
\]
we obtain on the attainable simplex tangent space
\[
I_{Q_0}(m_0+h)
=
\frac12h^\top\Sigma_p^\dagger h+o(\norm{h}^2),
\qquad
J_I=\Sigma_p^\dagger.
\]
Thus two proper scores can have the same Bayes-action map and the same accessibility metric while assigning different consequences to predictive errors.  For the Brier score, $G_S$ is Euclidean up to scale; for log score, it is the multinomial Fisher metric \cite{gneitingraftery2007strictly}.  The elicited functional fixes the action fibers, while the loss fixes the target-side regret geometry around them.

\subsection{Centered mean-field Gaussian approximation: a pure coverage defect}
\label{subsec:mf-covariance-example-v2}

Let $Q_0=N(0,I_2)$ and compare the centered full-Gaussian family
\[
\cQ_{\mathrm{full}}=\{N(0,\Sigma):\Sigma\succ0\}
\]
with the centered mean-field family
\[
\cQ_{\mathrm{MF}}
=\{N(0,\operatorname{diag}(\sigma_1^2,\sigma_2^2)):\sigma_1^2,\sigma_2^2>0\}.
\]
For a local covariance path $Q_t=N(0,I_2+tH)$ with $H=H^\top$, the score is
\[
s_H(\theta)
=
\frac12\{\theta^\top H\theta-\operatorname{tr}(H)\},
\qquad
\norm{s_H}_{L^2(Q_0)}^2
=
\frac12\operatorname{tr}(H^2).
\]
The centered mean-field tangent contains the two variance scores $\theta_1^2-1$ and $\theta_2^2-1$, but it does not contain the correlation score $\theta_1\theta_2$.

Consider the random second-moment vector
\[
\phi(\theta)
=
\begin{pmatrix}
\theta_1^2\\
\sqrt2\,\theta_1\theta_2\\
\theta_2^2
\end{pmatrix}
\]
and use weighted squared loss
\[
L(\theta,a)=(a-\phi(\theta))^\top W(a-\phi(\theta)),
\qquad
W=\operatorname{diag}(1,\omega,1),
\quad 0<\omega\le1.
\]
The action is $a\in\mathbb R^3$, and its Bayes value under $Q$ is $\Gamma(Q)=\E_Q\phi(\theta)$.  A covariance perturbation $H$ induces action displacement
\[
h=
\begin{pmatrix}
H_{11}\\
\sqrt2H_{12}\\
H_{22}
\end{pmatrix},
\qquad
\operatorname{tr}(H^2)=\norm h^2.
\]
Hence $J_I=\tfrac12I_3$ and $G=2W$, with generalized eigenvalues $4$, $4\omega$, and $4$.

The centered mean-field tangent image is $S_{\mathrm{MF}}=\operatorname{span}\{e_1,e_3\}$, so the correlation coordinate is missing and
\[
\eta_{\mathrm{MF}}=4\omega.
\]
Therefore
\[
\mathcal E_\rho(T_{\mathrm{MF}})=4\omega\rho,
\]
and centered mean-field gives the uniform local covariance-decision guarantee exactly when $\epsilon\ge4\omega\rho$.  The missing correlation direction matters only while its worst-case regret contribution lies above tolerance.

For $0<\omega<1$,
\[
d_{\mathrm{eff}}(\epsilon;\rho)
=
\begin{cases}
3,&0\le\epsilon<4\omega\rho,\\
2,&4\omega\rho\le\epsilon<4\rho,\\
0,&\epsilon\ge4\rho.
\end{cases}
\]
The mean-field image has dimension two and becomes sufficient exactly when the correlation-sensitive mode drops below tolerance.  The two retained diagonal directions are realized at the ambient minimum Fisher cost, so $\kappa_{\mathrm{MF}}=1$.  For this decision, mean field has a coverage/alignment defect without an information-efficiency defect.

Appendix~\ref{app:restricted-family-details} gives the Gaussian tangent calculation and an exact finite-Gaussian check showing that the local prediction $4\omega\rho$ agrees with the corresponding finite perturbation calculation to leading order.

\subsection{The same family can have complete coverage and excess information cost}
\label{subsec:mf-efficiency-example-v2}

Now change the decision while keeping the family fixed.  Consider the scalar statistic
\[
g_\alpha(\theta)
=
\theta_1\theta_2+\alpha(\theta_1^2-1),
\qquad \alpha\ne0,
\]
under squared-error loss for an action $a\in\mathbb R$.  Its Bayes action is $\Gamma(Q)=\E_Qg_\alpha(\theta)$.  At $Q_0=N(0,I_2)$,
\[
\operatorname{Var}(\theta_1\theta_2)=1,
\qquad
\operatorname{Var}(\theta_1^2-1)=2,
\qquad
\operatorname{Cov}(\theta_1\theta_2,\theta_1^2-1)=0.
\]
Thus $BB^*=1+2\alpha^2$ and $J_I=(1+2\alpha^2)^{-1}$.

Centered mean-field cannot use the correlation component of the ambient efficient score, but it can still move the scalar action through the marginal-variance component.  Projection onto its tangent space gives $B\Pi_{T_{\mathrm{MF}}}B^*=2\alpha^2>0$, so $B(T_{\mathrm{MF}})=\mathbb R$ and $\eta_{\mathrm{MF}}=0$.  The family has full first-order decision coverage for this scalar task.

The difference appears in Fisher energy.  For scalar displacement $h$,
\[
I_{\mathrm{full}}^{\mathrm{tan}}(h)
=
\frac{h^2}{2(1+2\alpha^2)},
\qquad
I_{\mathrm{MF}}^{\mathrm{tan}}(h)
=
\frac{h^2}{4\alpha^2},
\]
so
\[
\kappa_{\mathrm{MF}}
=
1+\frac1{2\alpha^2}>1.
\]
At $\alpha=1/\sqrt2$, $\kappa_{\mathrm{MF}}=2$.  Mean field can therefore realize every first-order scalar displacement, but it reaches those displacements by moving a marginal variance more strongly than the ambient minimum-information deformation requires.

The two mean-field examples summarize the distinction:
\[
\begin{array}{c|cc}
\text{decision} & \eta_{\mathrm{MF}} & \kappa_{\mathrm{MF}}\\
\hline
\text{weighted second-moment vector} & 4\omega & 1\\
\text{scalar mixed moment} & 0 & 1+1/(2\alpha^2).
\end{array}
\]
The local expressivity of centered mean-field is therefore decision-relative: one task exposes an unavailable correlation direction, while another task can use a marginal-variance direction as a more expensive substitute.

\subsection{A structured-Gaussian comparison: global KL and decision ranking can differ}
\label{subsec:structured-gaussian-comparison}

A small finite calculation illustrates how decision-relative ranking can differ from distribution-wide fidelity.  Consider a four-dimensional centered Gaussian target, representative of a regression posterior with two correlated coefficient pairs,
\[
P=N(0,\Sigma_P),
\qquad
\Sigma_P=
\begin{pmatrix}
1&0.70&0&0\\
0.70&1&0&0\\
0&0&1&0.35\\
0&0&0.35&1
\end{pmatrix}.
\]
Compare three centered Gaussian approximation families with unit marginal variances.  The mean-field family fixes both correlations at zero.  The family $\cQ_{12}$ permits only the $(1,2)$ correlation to vary, while $\cQ_{34}$ permits only the $(3,4)$ correlation to vary.  Because the target is block diagonal, the reverse-KL projections are obtained by matching the allowed target correlation and leaving the forbidden block independent:
\[
Q_{12}=N(0,\Sigma_{12}),
\qquad
Q_{34}=N(0,\Sigma_{34}),
\qquad
Q_{\mathrm{MF}}=N(0,I_4),
\]
where $\Sigma_{12}$ keeps correlation $0.70$ and sets the $(3,4)$ correlation to zero, while $\Sigma_{34}$ does the reverse.

For a centered Gaussian $Q=N(0,\Sigma_Q)$,
\[
\KL(Q\|P)
=
\frac12\left\{
\operatorname{tr}(\Sigma_P^{-1}\Sigma_Q)-4
+\log\frac{\det\Sigma_P}{\det\Sigma_Q}
\right\}.
\]
This gives
\[
\KL(Q_{12}\|P)=0.0743,
\qquad
\KL(Q_{34}\|P)=0.6241,
\qquad
\KL(Q_{\mathrm{MF}}\|P)=0.6984.
\]
By global reverse KL, $Q_{12}$ is the clearly superior structured approximation because it captures the stronger correlation.

Now consider two uncertainty-sensitive decisions.  Let
\[
z_{12}=2^{-1/2}(1,1,0,0)^\top,
\qquad
z_{34}=2^{-1/2}(0,0,1,1)^\top,
\]
and, for $j\in\{12,34\}$, use squared-error loss for the statistic
\[
g_j(\theta)=(z_j^\top\theta)^2.
\]
The Bayes action is the posterior second moment $a_j(Q)=\E_Q g_j=z_j^\top\Sigma_Qz_j$.  Under the target,
\[
a_{12}(P)=1.70,
\qquad
a_{34}(P)=1.35.
\]
The resulting target regrets of the three approximations are
\begin{center}
\begin{tabular}{lccc}
\toprule
approximation & $\KL(Q\|P)$ & $\Delta_{12}$ & $\Delta_{34}$\\
\midrule
$Q_{12}$ & $0.0743$ & $0$ & $0.1225$\\
$Q_{34}$ & $0.6241$ & $0.4900$ & $0$\\
$Q_{\mathrm{MF}}$ & $0.6984$ & $0.4900$ & $0.1225$\\
\bottomrule
\end{tabular}
\end{center}
The globally better structured approximation $Q_{12}$ is exact for the first decision, while the globally worse $Q_{34}$ is exact for the second.  The example is deliberately small, but it shows the operational distinction directly: distribution-wide KL fidelity supplies one ordering of approximation families, whereas a specified decision problem can supply another.  The calculations are reproduced by the script \texttt{experiments/decision\_relative\_gaussian\_comparison.py} in the accompanying repository.

%% file: discussion.tex
\section{Discussion}
\label{sec:discussion}

A specified decision problem induces three coupled objects for posterior approximation.  The target posterior and loss determine a regret geometry on actions; the baseline approximation determines the information required to move through that geometry; and a restricted approximation family determines which movements are available and at what Fisher cost.  In regular problems, the pair $(G,J_I)$ makes the first two objects explicit, while $(\eta_T,\kappa_T)$ summarize local family coverage and information efficiency.

\subsection{Implications for variational approximation}

The geometry gives a constructive criterion for comparing and extending variational families.  A useful family extension can improve decision support in two ways: it can add an omitted action direction with substantial regret consequence, thereby reducing $\eta_T$, or it can add a posterior score direction that lowers the Fisher cost of an already-reachable action displacement, thereby reducing $\kappa_T$.  These two improvements are distinct even when the candidate families have the same parameter count or tangent-image rank.

The oracle effective dimension $d_{\mathrm{eff}}(\epsilon;\rho)$ gives the smallest rank that could support a uniform decision guarantee at a chosen regret and information scale.  The coverage defect $\eta_T$ then measures how well the actual family image is aligned with that oracle geometry.  When coverage is complete, $\kappa_T$ determines whether the family reaches the relevant directions with the same local information efficiency as the ambient posterior class.

The examples illustrate why this decision-relative view can change a family comparison.  The same centered mean-field family can have a pure coverage defect for one decision and complete coverage with excess Fisher cost for another.  In the structured-Gaussian example, the family with the better global reverse-KL approximation is not the family that best preserves every downstream second-moment decision.  These are different questions because they weight posterior discrepancies through different operational maps.

\subsection{Tolerance and information scale}

The effective dimension is indexed by the regret tolerance $\epsilon$ and information scale $\rho$.  Tightening the regret budget can make previously negligible directions relevant, while increasing the posterior perturbation radius can expose additional consequential directions.  The family certificate $\rho\eta_T\le\epsilon$ has the same dependence.

The choice of $\rho$ is therefore part of the application-level specification of the approximation problem.  It can be set by a scientifically meaningful KL neighborhood, by the observed scale of a target--baseline mismatch, or by a distribution of plausible perturbations.  When a target ensemble is available, the average-case representation in Section~\ref{subsec:ensemble-dimension} replaces worst-case structural importance by the frequency with which plausible mismatches occupy each direction.

Structural and realized importance should be kept separate.  A large generalized eigenvalue means that a direction can generate substantial regret per unit information; it does not imply that a particular target mismatch has a large component in that direction.

\subsection{Scope of the local theory}

The quadratic theory covers regular finite-dimensional decisions with nondegenerate first-order response and enough localization to keep the relevant optimization near the baseline.  The global regret and information contractions remain meaningful beyond that regime, but their local approximations can change.  Binary classification already gives a simple nonquadratic example because the acceptable action set changes discontinuously.

The restricted-family diagnostics are also local.  Curved, disconnected, or irregular families can admit finite movements that are invisible from a single linear tangent space, and favorable tangent geometry does not imply that a nonconvex optimizer will find the relevant family member.  The exact restricted contraction remains the global benchmark; $(\eta_T,\kappa_T)$ describe the first-order geometry around $Q_0$ for regular families.

The ambient I-projection is likewise an information-theoretic benchmark rather than a claim that every minimizing distribution can arise from the original prior--likelihood model.  Model realizability, family restriction, optimization error, Monte Carlo error, and finite computation impose additional layers of constraint.

Finally, the present development treats $G$, $J_I$, and the response operators as known.  Applied use requires estimating these objects, and uncertainty in those estimates can change the apparent effective dimension, mode ordering, or family defects.  Quantifying that uncertainty is a statistical problem built on top of the geometry developed here.

\subsection{Directions for extension}

The geometry suggests a direct family-design problem.  When $\eta_T$ is driven by a small number of omitted high-consequence modes, one can seek the smallest tangent-space augmentation that captures them.  When $\eta_T=0$ but $\kappa_T$ is large, one can instead seek posterior directions that recover the ambient minimum-information lifts of reachable decisions.  This turns decision geometry into a guide for structured variational-family design.

There is also a dynamic version of the problem.  In an iterative approximation algorithm, the current approximation can be viewed as a state, admissible family or algorithm updates as controls, the induced decision displacement as an output, and KL or computational effort as a movement cost.  This formulation leads naturally to reachability and minimum-energy questions under algorithmic constraints and provides a concrete point of contact with control theory.

Beyond the regular regime, boundary decisions can lead to tangent-cone geometry, while degenerate or singular decision maps can require higher-order local structure rather than the quadratic pair $(G,J_I)$.  The global decision problem remains the same; the appropriate local approximation changes with the order and shape of the response.

\subsection{Conclusion}

For a specified downstream decision problem, posterior approximation can be assessed through three quantities: the regret induced by action error, the information required to induce that error from a baseline approximation, and the ability of a chosen approximation family to realize the corresponding decision directions.  This decision-relative organization yields exact global accessibility measures, local decision--information modes, a tolerance-dependent effective dimension, and separate diagnostics for family coverage and information efficiency.

%% file: appendix-regularity.tex
\section{Regularity Conditions for the Local Theory}
\label{app:regularity-v2}

The global regret and information contractions require relatively little local smoothness.  The quadratic theory in Sections~\ref{sec:local-geometry}--\ref{sec:restricted-families} uses additional regularity for three distinct purposes: identifying Bayes-action fibers by finite-dimensional equations, differentiating the Bayes-action map in Fisher score space, and localizing the moving constrained optimization problems.  This appendix records those assumptions by role.

\subsection{Equality-defined action fibers and exact I-projection}
\label{app:regularity-identification}

Let $\cA\subset\mathbb R^d$ be an open local action chart, let $Q_0$ be a baseline distribution, and let $a_0$ denote its unique Bayes action.  The exact contracted-information calculation uses the following conditions.

\begin{enumerate}[label=(I\arabic*)]
\item \textbf{Exact finite-information identification.}  There is a neighborhood $U$ of $a_0$ and a measurable map $\psi_a:\Theta\to\mathbb R^d$ such that, for every $a\in U$ and every $Q\in\mathfrak C_{Q_0}$ with $\KL(Q\|Q_0)<\infty$,
\[
a\in\cB_0(Q)
\quad\Longleftrightarrow\quad
\E_Q\psi_a=0.
\]
Thus the estimating equation characterizes the finite-information portion of the Bayes-action fiber exactly, rather than serving only as a stationarity condition.

\item \textbf{Smooth local log-partition map and admissible tilts.}  Define
\[
\Lambda(a,\lambda)
=\log\E_{Q_0}\exp\{\lambda^\top\psi_a\}.
\]
There are neighborhoods $U_0\subset U$ of $a_0$ and $L_0$ of $0\in\mathbb R^d$ such that $\Lambda$ is finite and $C^2$ on $U_0\times L_0$.  For every $(a,\lambda)$ in this neighborhood, the exponential tilt
\[
\frac{dQ_{a,\lambda}}{dQ_0}
=
\exp\left\{\lambda^\top\psi_a-\Lambda(a,\lambda)\right\}
\]
belongs to $\mathfrak C_{Q_0}$ (and has well-defined Bayes risk).

\item \textbf{Baseline identification.}  $\E_{Q_0}\psi_{a_0}=0$.

\item \textbf{Nondegenerate identification covariance.}  
\[
V:=\operatorname{Cov}_{Q_0}(\psi_{a_0})\succ0.
\]

\item \textbf{Nonsingular action sensitivity.}  
\[
H:=D_a[\E_{Q_0}\psi_a]_{a=a_0}
\]
is nonsingular.
\end{enumerate}

For a smooth loss, one sufficient route to (I1) is that every finite-information risk $R_Q(a)$ under consideration is differentiable and strictly convex on a convex action domain with an interior minimizer, with $\psi_a(\theta)=\nabla_aL(\theta,a)$.  The first-order condition then characterizes the unique global Bayes action.  The identification-function formulation also covers regular nonsmooth examples such as quantiles.

Assumption (I2) is stronger than a second-moment condition.  It supplies the local exponential moments needed for the exact KL minimizer and guarantees that the minimizing exponential tilts belong to the same ambient candidate class over which the KL contraction is defined.  The Hilbert-space Fisher geometry below uses only the corresponding pathwise differentiability and second-moment structure.

\subsection{Pathwise differentiability and Fisher geometry}
\label{app:regularity-pathwise}

Let $\mathcal H=L_0^2(Q_0)$ be the ambient score space.  For a regular path with
\[
\frac{dQ_t}{dQ_0}=1+t s+o_{L^2}(t),
\qquad s\in\mathcal H,
\]
and an action path $a_t=a_0+t h+o(t)$, assume the identifying equation has the expansion
\begin{equation}
\label{eq:app-pathwise-id-expansion}
\E_{Q_t}\psi_{a_t}
=
t\left\{
Hh+\E_{Q_0}[s\psi_{a_0}]
\right\}+o(t).
\end{equation}
Suitable $L^2$ continuity or differentiability of $a\mapsto\psi_a$, together with domination or uniform integrability sufficient to interchange differentiation and expectation, gives one route to this expansion.

The Bayes-action functional is then pathwise differentiable with derivative
\[
B[s]
=-H^{-1}\E_{Q_0}[s\psi_{a_0}].
\]
For the linearized theory, define the effective action space by
\[
E:=\operatorname{ran}B.
\]
The map $B:\mathcal H\to E$ is onto by construction.  Under (I4)--(I5), the covariance operator of $s\mapsto\E_{Q_0}[s\psi_{a_0}]$ is $V\succ0$, so in the regular equality-defined setting $E=\mathbb R^d$.  More generally, the Fisher geometry can be formulated directly on the finite-dimensional range $E$ whenever $B$ is bounded.

On $E$, the minimum-score-energy problem
\[
\inf_{s:B[s]=h}\frac12\norm{s}_{L^2(Q_0)}^2
\]
is governed by
\[
J_I=(BB^*)^{-1}.
\]
This pathwise construction requires only the bounded derivative $B$ and the relevant second moments; it can therefore remain available when the exponential-moment condition in (I2) fails.

\subsection{A checkable sufficient condition for smooth losses}
\label{app:regularity-smooth-loss}

The identification-function assumptions above are the primitive conditions used by the theory.  When the Bayes action is defined by a smooth loss, a more familiar sufficient condition can be stated directly in terms of $L$.  Let $A_0$ be a neighborhood of the baseline Bayes action $a_0$.

\begin{enumerate}[label=(S\arabic*)]
\item \textbf{Local smoothness and an interior optimum.}  The action $a_0$ is the unique interior minimizer of $R_{Q_0}$, and for $Q_0$-almost every $\theta$ the map $a\mapsto L(\theta,a)$ is twice continuously differentiable on $A_0$.

\item \textbf{Differentiation under the baseline expectation.}  There are integrable envelopes sufficient to interchange differentiation and expectation on $A_0$; for example,
\[
\E_{Q_0}\left[
\sup_{a\in A_0}\norm{\nabla_aL(\theta,a)}^2
\right]<\infty,
\qquad
\E_{Q_0}\left[
\sup_{a\in A_0}\norm{\nabla_a^2L(\theta,a)}_{\mathrm{op}}
\right]<\infty.
\]
These conditions imply that $\psi_a(\theta)=\nabla_aL(\theta,a)$ is locally $L^2(Q_0)$ and that
\[
H
=
D_a\E_{Q_0}[\psi_a]_{a=a_0}
=
\E_{Q_0}[\nabla_a^2L(\theta,a_0)].
\]

\item \textbf{Nondegenerate decision response.}  The matrix $H$ is nonsingular and
\[
V
=
\operatorname{Cov}_{Q_0}\{\nabla_aL(\theta,a_0)\}
\]
is positive definite on the effective action coordinates.

\item \textbf{Stable Bayes branch along regular paths.}  For each regular path $Q_t$ with score $s$ under consideration, assume that
\[
F(t,a):=\E_{Q_t}[\nabla_aL(\theta,a)]
\]
is continuously differentiable near $(0,a_0)$, with
\[
D_aF(0,a_0)=H,
\qquad
D_tF(0,a_0)=\E_{Q_0}[s\,\nabla_aL(\theta,a_0)].
\]
Assume also that, for sufficiently small $t$, a global Bayes action under $Q_t$ exists, is interior, and lies in the local neighborhood on which the implicit-function theorem gives the unique stationary branch.  This last requirement follows, for example, from uniform convergence of the risks along the path and a strict global risk gap away from $a_0$, provided the Bayes actions exist.  Under (S1)--(S4), the global Bayes-action derivative along these paths is
\[
B[s]
=
-H^{-1}\E_{Q_0}
\left[s\,\nabla_aL(\theta,a_0)\right],
\]
and the induced information metric is
\[
J_I=H^\top V^{-1}H=(BB^*)^{-1}.
\]

\item \textbf{Regular target regret.}  At the target posterior $P$, assume the analogous local differentiation conditions around its unique interior Bayes action $a_P$, together with isolation of the optimum and
\[
G
=
\nabla_a^2R_P(a_P)
=
\E_P[\nabla_a^2L(\theta,a_P)]
\succ0.
\]
Then the target regret has the quadratic germ
\[
\Delta_P(a_P+h)
=
\frac12h^\top Gh+o(\norm{h}^2).
\]
\end{enumerate}

Conditions (S1)--(S4) are sufficient for the local Bayes-action derivative and Fisher quotient geometry along the specified regular paths; (S5) addresses the target-side regret expansion.  To identify the exact finite-information Bayes fibers and derive the exact exponential I-projection, one additionally needs the stronger global-in-$Q$ identification and local exponential-integrability requirements in (I1)--(I2).  A convenient sufficient route is that every finite-information risk under consideration is strictly convex with an interior minimizer and that the local moment-generating function of $\nabla_aL(\theta,a)$ is finite near the origin.

Smoothness of the pointwise loss is not necessary.  The quantile example uses a nonsmooth pinball loss but admits a regular identification function with nonsingular derivative and finite variance, so the same local geometry follows.  The essential requirements are regularity of the Bayes-action functional and a nondegenerate quadratic target regret, rather than $C^2$ smoothness of $L$ itself.

\subsection{Target-side regularity}
\label{app:regularity-target}

For a fixed target posterior $P$, the local ellipsoidal form of the regret filtration uses the following assumptions.

\begin{enumerate}[label=(T\arabic*)]
\item The Bayes action $a_P$ is unique and lies in the interior of the action chart $\cA$.
\item The risk $R_P$ is $C^2$ on a neighborhood of $a_P$.
\item The regret Hessian
\[
G=\nabla_a^2R_P(a_P)
\]
is positive definite.
\item The optimum is isolated in regret: there is a neighborhood $V_P$ of $a_P$, with $\overline V_P$ contained in the local chart, such that
\[
\inf_{a\in\cA\setminus V_P}\Delta_P(a)>0.
\]
\end{enumerate}

The isolation condition promotes the local quadratic germ at $a_P$ to the entire sufficiently small regret sublevel set.  Without global isolation, the same quadratic calculation describes the connected local component containing $a_P$.

\subsection{Moving-target assumptions for repair and failure}
\label{app:regularity-moving-target}

Let $P_t$ be a sequence of targets with risks $R_t$, regrets $\Delta_t$, and unique Bayes actions $a_t$.  The local repair and robustness theorem uses the following finite-dimensional conditions.

\begin{enumerate}[label=(M\arabic*)]
\item \textbf{Local target displacement.}
\[
\frac{a_t-a_0}{t}\to\delta\in\mathbb R^d.
\]

\item \textbf{Regular contracted information.}  The exact value function $I_{Q_0}$ is $C^2$ on a neighborhood of $a_0$ and
\[
I_{Q_0}(a_0)=0,
\qquad
\nabla I_{Q_0}(a_0)=0,
\qquad
\nabla^2I_{Q_0}(a_0)=J_I\succ0.
\]
Conditions (I1)--(I5) provide one sufficient route to this assumption.

\item \textbf{Uniform target Hessian stability.}  For every $C<\infty$,
\[
\sup_{\|u\|\le C}
\left\|
\nabla^2R_t(a_t+t u)-G
\right\|_{\mathrm{op}}
\longrightarrow0
\]
for some $G\succ0$.

\item \textbf{Target and information localization.}  There exist $r,c_R,\kappa_R>0$ such that, for all sufficiently small $t$,
\[
\Delta_t(a)\ge c_R\|a-a_t\|^2
\qquad
\text{when }\|a-a_t\|\le r,
\]
and
\[
\inf_{\|a-a_t\|\ge r}\Delta_t(a)\ge\kappa_R.
\]
Choose also a sufficiently small fixed neighborhood $U$ of $a_0$ on which (M2) yields
\[
I_{Q_0}(a)\ge c_I\|a-a_0\|^2
\qquad (a\in U)
\]
for some $c_I>0$, and assume the global information separation
\[
\inf_{a\notin U}I_{Q_0}(a)\ge\kappa_I>0.
\]
\end{enumerate}

The two parts of (M4) play different localization roles.  The regret bounds keep the rescaled repair sets in a bounded neighborhood of the moving target action.  The information gap keeps $O(t^2)$ failure near-minimizers inside the fixed neighborhood $U$; the quadratic lower bound there then forces $\|a-a_0\|=O(t)$.  Together they rule out remote branches whose exact objective is as small as the local $O(t^2)$ scale.

Under (M1)--(M4), uniformly for $h$ in compact sets,
\[
I_{Q_0}(a_0+t h)
=
\frac{t^2}{2}h^\top J_Ih+o(t^2),
\qquad
\Delta_t(a_0+t h)
=
\frac{t^2}{2}(h-\delta)^\top G(h-\delta)+o(t^2).
\]
These are the two expansions used in Theorem~\ref{thm:local-repair-robustness-v2}.

\subsection{Restricted-family assumptions}
\label{app:regularity-family}

For the restricted-family diagnostics, let $E=\operatorname{ran}B$ be finite-dimensional.  Write $\mathcal D_{\cQ}(Q_0)$ for the set of scores of regular paths lying in $\cQ$ through $Q_0$.  The regular linear-family theory assumes that $\mathcal D_{\cQ}(Q_0)$ is a linear subspace (possibly nonclosed), and sets
\[
T:=\overline{\mathcal D_{\cQ}(Q_0)}^{\,L^2(Q_0)}\subset\mathcal H.
\]
The closure encodes approximability by feasible path scores, without declaring arbitrary linear combinations of unrelated feasible directions to be attainable.  As $B$ is bounded and $E$ finite-dimensional, $B(T)$ is the closure of the first-order action images of such paths.  In particular, this assumption excludes a family with a union of tangent axes whose linear span would contain unrealizable diagonal directions.  Its decision image
\[
S_T=B(T)
\]
is finite-dimensional and therefore closed, and the restricted score-energy operator is $B\Pi_TB^*$.

These statements concern the linearized family at $Q_0$.  A nonlinear or irregular parameterization can have finite movements not represented by one tangent space, and a tangent vector need not imply a corresponding finite global move.  Identifying the quadratic form based on $B\Pi_TB^*$ with the Hessian of the exact restricted contraction $I_{Q_0,\cQ}$ requires additional smooth path-realizability, second-order KL expansion, and localization conditions.

\subsection{Large-deviation qualification}
\label{app:regularity-sanov}

The large-deviation interpretation in Section~\ref{sec:posterior-accessibility} concerns the same deterministic KL variational problems; it supplies a probabilistic interpretation of the forward-KL radius rather than a model of approximation error.  Write
\[
\cJ_{Q_0}(Q)=\KL(Q\|Q_0)
\]
for the measure-level rate appearing in the i.i.d. case.  If
\[
Q_0^{(n)}=Q_0^{\otimes n}
\]
is the $n$-sample law and $\widehat Q_n$ is the empirical measure, Sanov's theorem supplies $\cJ_{Q_0}$ as the empirical-measure rate function under its standard topological hypotheses \cite{dembozeitouni1998}.

Sanov's principle applies on the empirical-measure state space (for instance, probability measures on a Polish sample space with the weak topology), not automatically on the candidate class $\mathfrak C_{Q_0}$.  In particular, if $Q_0$ is nonatomic and $\mathfrak C_{Q_0}$ requires finite $\KL(Q\|Q_0)$, every empirical measure $\widehat Q_n$ lies outside $\mathfrak C_{Q_0}$ almost surely.  The event in the probability statement must therefore be defined on the Sanov state space, as in the main text, rather than by restricting it to candidate distributions.

For a measurable empirical-measure event $A$, Sanov's upper and lower bounds yield
\[
-\inf_{\mu\in A^\circ}\cJ_{Q_0}(\mu)
\le
\liminf_{n\to\infty}\frac1n\log\Pr\{\widehat Q_n\in A\}
\le
\limsup_{n\to\infty}\frac1n\log\Pr\{\widehat Q_n\in A\}
\le
-\inf_{\mu\in\overline A}\cJ_{Q_0}(\mu).
\]
When the infima over the interior and closure agree, $A$ is a $\cJ_{Q_0}$-continuity set, and the event has exponent $\inf_{\mu\in A}\cJ_{Q_0}(\mu)$.  To identify that exponent with $\cR$ or $\cM$, the candidate-class restriction must additionally leave the corresponding rate infimum unchanged.  The empirical Bayes correspondence must also be well defined on the event; a direct contraction-principle derivation needs continuity, or an appropriate extended-contraction condition, of the action map.  Quantiles, ties, and set-valued decisions may violate such continuity assumptions without affecting the deterministic KL optimization.  For dependent samples, $Q_0^{(n)}$ denotes the general $n$-sample law and a suitable non-Sanov rate function may be required.

\subsection{Why the proof expands the contracted value function}

The regular proof uses the finite-dimensional value function $I_{Q_0}(a)$ rather than a uniform KL expansion over an unrestricted $L^2$ score ball.  Bounded subsets of $L_0^2(Q_0)$ can contain scores with spikes of height $O(t^{-1})$ on events of probability $O(t^2)$; such scores remain $L^2$ bounded while failing to be uniformly infinitesimal pointwise.  The $C^2$ expansion of the exact contracted-information value supplied by (I1)--(I5) avoids this obstruction and provides the uniform-on-compact action-space expansion needed by the moving optimization problems.

%% file: appendix-contracted-information.tex
\section{Exact Contracted Information and the Local Information Metric}
\label{app:contracted-info-v2}

This appendix derives the baseline-side metric used in Section~\ref{sec:local-geometry}.  Under (I1)--(I5), an action near $a_0$ is characterized on the finite-information part of its Bayes fiber by the equality constraint $\E_Q\psi_a=0$.  The KL contraction can therefore be solved by an exponential tilt, and its Hessian agrees with the inverse covariance geometry induced by the pathwise Bayes-action derivative.

\subsection{Multiplier map and exact I-projection}

Write
\[
\Lambda(a,\lambda)
=
\log\E_{Q_0}\exp\{\lambda^\top\psi_a\},
\qquad
F(a,\lambda)=\nabla_\lambda\Lambda(a,\lambda).
\]
At the baseline,
\[
F(a_0,0)=\E_{Q_0}\psi_{a_0}=0,
\qquad
D_\lambda F(a_0,0)=V\succ0.
\]
The implicit-function theorem therefore gives a neighborhood of $a_0$ and a unique $C^1$ multiplier map $\lambda(a)$, with $\lambda(a_0)=0$, such that
\[
\nabla_\lambda\Lambda(a,\lambda(a))=0.
\]
Differentiating at $a_0$ gives
\begin{equation}
\label{eq:app-v2-lambda-derivative}
D\lambda(a_0)=-V^{-1}H.
\end{equation}

For $a$ in this neighborhood, define the tilted law
\begin{equation}
\label{eq:app-v2-exponential-tilt}
\frac{dQ_a^*}{dQ_0}
=
\exp\left\{
\lambda(a)^\top\psi_a
-
\Lambda(a,\lambda(a))
\right\}.
\end{equation}
By construction,
\[
\E_{Q_a^*}\psi_a
=
\nabla_\lambda\Lambda(a,\lambda(a))
=0.
\]
Assumption (I2) makes $Q_a^*$ decision-admissible, and (I1) therefore places it in the Bayes-action fiber $\cC_a$.  Its KL divergence is
\[
\KL(Q_a^*\|Q_0)
=
-\Lambda(a,\lambda(a)),
\]
because the tilted expectation of $\psi_a$ vanishes.

\begin{proposition}[Exact local I-projection]
\label{prop:app-v2-exact-projection}
Under (I1)--(I5), for every $a$ sufficiently close to $a_0$,
\[
Q_a^*\in\arg\min_{Q\in\cC_a}\KL(Q\|Q_0)
\]
and
\begin{equation}
\label{eq:app-v2-contracted-value}
I_{Q_0}(a)
=
-\Lambda(a,\lambda(a)).
\end{equation}
\end{proposition}

\begin{proof}
Let $Q\in\cC_a$.  Infinite-KL competitors cannot improve on the finite value attained by $Q_a^*$.  If $\KL(Q\|Q_0)<\infty$, assumption (I1) gives $\E_Q\psi_a=0$.  The entropy variational inequality gives, for every admissible multiplier $\lambda$,
\[
\KL(Q\|Q_0)
\ge
\lambda^\top\E_Q\psi_a-\Lambda(a,\lambda)
=
-\Lambda(a,\lambda).
\]
Taking $\lambda=\lambda(a)$ gives the lower bound $-\Lambda(a,\lambda(a))$, and the tilted law in \eqref{eq:app-v2-exponential-tilt} attains it.
\end{proof}

The proposition is local in the action coordinate but global over the ambient distributional fiber: once $a$ is close enough to $a_0$ for the multiplier construction to apply, the minimization is over the full Bayes fiber rather than only over nearby distributions.

\subsection{Hessian of the contracted-information value}

Because $\Lambda(a,0)=0$ for every $a$, all pure $a$-derivatives of $\Lambda(a,0)$ vanish.  In particular,
\[
\nabla_a\Lambda(a_0,0)=0,
\qquad
\nabla_{aa}^2\Lambda(a_0,0)=0.
\]
Define
\[
I(a):=I_{Q_0}(a)=-\Lambda(a,\lambda(a)).
\]
The multiplier equation gives the envelope identity
\[
\nabla I(a)
=-\nabla_a\Lambda(a,\lambda(a)),
\]
so $\nabla I(a_0)=0$.  At $(a_0,0)$,
\[
\nabla_{a\lambda}^2\Lambda(a_0,0)=H^\top.
\]
Differentiating the envelope identity and using \eqref{eq:app-v2-lambda-derivative} therefore yields
\begin{align*}
\nabla^2I(a_0)
&=-\nabla_{aa}^2\Lambda(a_0,0)
-\nabla_{a\lambda}^2\Lambda(a_0,0)D\lambda(a_0)\\
&=H^\top V^{-1}H.
\end{align*}
Hence
\begin{equation}
\label{eq:app-v2-JI}
J_I
:=
\nabla^2I_{Q_0}(a_0)
=
H^\top V^{-1}H
\succ0.
\end{equation}
The $C^2$ value function consequently has the expansion
\[
I_{Q_0}(a_0+t h)
=
\frac{t^2}{2}h^\top J_Ih+o(t^2)
\]
uniformly for $h$ in compact subsets of the local action chart.

\subsection{Identification with the pathwise decision derivative}

Assume additionally the pathwise expansion \eqref{eq:app-pathwise-id-expansion}.  If
\[
\Gamma(Q_t)=a_0+t h+o(t)
\]
for a regular path with score $s$, differentiating the identifying equation gives
\[
Hh+\E_{Q_0}[s\psi_{a_0}]=0.
\]
Thus
\[
B[s]
=-H^{-1}\E_{Q_0}[s\psi_{a_0}].
\]
Let $\mathsf M:\mathcal H\to\mathbb R^d$ be
\[
\mathsf M[s]=\E_{Q_0}[s\psi_{a_0}].
\]
Then
\[
\mathsf M\mathsf M^*=V,
\qquad
B=-H^{-1}\mathsf M,
\]
so
\[
BB^*=H^{-1}VH^{-\top}.
\]
Combining this identity with \eqref{eq:app-v2-JI} gives
\[
J_I=(BB^*)^{-1}.
\]
The Hessian of the exact contracted KL value is therefore the inverse covariance metric induced by the pathwise Bayes-action derivative.

\subsection{Minimum-information lift and the first-order projection path}

Let
\[
E=\operatorname{ran}B
\]
be the effective action space.  For a desired displacement $h\in E$, consider
\[
\min_{s\in\mathcal H}
\left\{
\frac12\norm{s}_{L^2(Q_0)}^2:
B[s]=h
\right\}.
\]
The unique minimum-norm solution is
\[
s_h^*
=B^*(BB^*)^{-1}h
=B^*J_Ih,
\]
with energy
\[
\frac12\norm{s_h^*}_{L^2(Q_0)}^2
=
\frac12h^\top J_Ih.
\]
Equation \eqref{eq:app-v2-lambda-derivative} gives
\[
\lambda(a_0+t h)
=-tV^{-1}Hh+o(t).
\]
If the tilted densities in \eqref{eq:app-v2-exponential-tilt} are differentiable in $L^2(Q_0)$ at the baseline, then
\[
\frac{dQ_{a_0+t h}^*}{dQ_0}
=
1+t s_h^*+o_{L^2}(t),
\]
with the explicit score
\[
s_h^*(\theta)
=-\psi_{a_0}(\theta)^\top V^{-1}Hh.
\]
Thus the horizontal minimum-energy lift used in the main text is the first-order tangent to the exact exponential I-projection.

\subsection{Canonical checks}

\paragraph{Mean.}
For squared loss, take $\psi_a=a-g(\theta)$.  Then
\[
H=1,
\qquad
V=\operatorname{Var}_{Q_0}(g),
\qquad
J_I=\operatorname{Var}_{Q_0}(g)^{-1}.
\]

\paragraph{Quantile.}
Consider pinball loss at level $\tau\in(0,1)$ and suppose $Q_0$ has continuous CDF $F_0$ with $\E_{Q_0}|Y|<\infty$.  For an arbitrary distribution $Q$, the Bayes condition for action $a$ is
\[
F_Q(a^-)\le\tau\le F_Q(a).
\]
If $\KL(Q\|Q_0)<\infty$, then $Q\ll Q_0$.  Continuity of $F_0$ therefore implies $Q(\{a\})=0$, and the Bayes condition reduces to $F_Q(a)=\tau$.  The finite-information Bayes fiber is exactly the equality constraint generated by
\[
\psi_a(Y)=\mathbf1\{Y\le a\}-\tau.
\]

Let $A_a=(-\infty,a]$ and $p=F_0(a)$.  For $0<p<1$, every finite-KL feasible $Q$ satisfies $Q(A_a)=\tau$, and the KL chain rule gives
\begin{align*}
\KL(Q\|Q_0)
&=
\operatorname{kl}(\tau\|p)\\
&\quad+
\tau\KL\!\left(Q(\cdot\mid A_a)\,\middle\|\,Q_0(\cdot\mid A_a)\right)\\
&\quad+
(1-\tau)\KL\!\left(Q(\cdot\mid A_a^c)\,\middle\|\,Q_0(\cdot\mid A_a^c)\right),
\end{align*}
where
\[
\operatorname{kl}(u\|v)
=u\log\frac uv+(1-u)\log\frac{1-u}{1-v}
\]
is Bernoulli KL.  The minimum preserves the baseline conditional laws on the two sides of $a$ and changes only their weights:
\[
\frac{dQ_a^*}{dQ_0}(y)
=
\frac{\tau}{F_0(a)}\mathbf1\{y\le a\}
+
\frac{1-\tau}{1-F_0(a)}\mathbf1\{y>a\}.
\]
The density ratio is bounded for fixed $a$ with $0<F_0(a)<1$, so $Q_a^*$ inherits finite first moment from $Q_0$.  Consequently
\[
I_{Q_0}(a)
=
\operatorname{kl}\!\left(\tau\,\middle\|\,F_0(a)\right),
\]
with value $+\infty$ when $F_0(a)\in\{0,1\}$.  If $a_0$ satisfies $F_0(a_0)=\tau$ and $F_0$ has a positive sufficiently smooth density $f_0$ near $a_0$, then
\[
H=f_0(a_0),
\qquad
V=\tau(1-\tau),
\qquad
J_I
=
\frac{f_0(a_0)^2}{\tau(1-\tau)}.
\]
The same value follows by differentiating the binary-KL contraction twice at $a_0$.

\paragraph{Finite-outcome proper prediction.}
Work in effective simplex coordinates and take
\[
\psi_q(\theta)=p(\theta)-q.
\]
Then
\[
H=-I,
\qquad
V=\operatorname{Cov}_{Q_0}(p(\theta)),
\]
so $J_I=V^{-1}$ on the effective tangent space, or equivalently $V^\dagger$ in redundant simplex coordinates.  Since $p(\theta)$ is bounded, the local exponential-integrability requirement is automatic.

%% file: appendix-local-proofs.tex
\section{Proofs for the Local Repair and Robustness Limits}
\label{app:local-proofs-v2}

This appendix proves Theorem~\ref{thm:local-repair-robustness-v2} and records the generalized-mode KKT formulas used in Section~\ref{sec:local-geometry}.  The key technical issue is localization: uniform quadratic expansions identify the limiting optimization only after near-minimizers are shown to remain inside the $O(t)$ neighborhood where those expansions apply.

\subsection{Regular blow-up of regret sublevel sets}

Let $P$ have a unique isolated interior Bayes action $a_P$, let $R_P$ be $C^2$ near $a_P$, and let
\[
G=\nabla_a^2R_P(a_P)\succ0.
\]
Fix $\eta\in(0,\lambda_{\min}(G))$.  By continuity of the Hessian,
\[
G-\eta I
\preceq
\nabla_a^2R_P(a_P+h)
\preceq
G+\eta I
\]
for all sufficiently small $h$.  Taylor's theorem in integral form and $\nabla R_P(a_P)=0$ give
\begin{equation}
\label{eq:app-v2-regret-sandwich}
\frac12h^\top(G-\eta I)h
\le
\Delta_P(a_P+h)
\le
\frac12h^\top(G+\eta I)h.
\end{equation}
The isolation condition in Appendix~\ref{app:regularity-target} forces the entire sufficiently small regret sublevel set into this local chart.  Setting $h=\sqrt\epsilon\,u$ in \eqref{eq:app-v2-regret-sandwich} sandwiches the rescaled acceptable set between
\[
\left\{u:\frac12u^\top(G+\eta I)u\le1\right\}
\quad\text{and}\quad
\left\{u:\frac12u^\top(G-\eta I)u\le1\right\}.
\]
Letting first $\epsilon\downarrow0$ and then $\eta\downarrow0$ yields
\[
\frac{\cB_\epsilon(P)-a_P}{\sqrt\epsilon}
\xrightarrow{d_H}
\left\{u:\frac12u^\top Gu\le1\right\}.
\]
Without global isolation, the same argument applies to the connected local component containing $a_P$.

\subsection{A compact optimization lemma}

\begin{lemma}[Stability of constrained minima]
\label{lem:app-v2-compact-minima}
Let $F_t,F\subset\mathbb R^d$ be nonempty compact sets contained in a common compact set $K$, and suppose
\[
F_t\xrightarrow{d_H}F.
\]
Let $\phi_t,\phi:K\to\mathbb R$ be continuous and satisfy
\[
\sup_{h\in K}|\phi_t(h)-\phi(h)|\to0.
\]
Then
\[
\inf_{h\in F_t}\phi_t(h)
\longrightarrow
\inf_{h\in F}\phi(h).
\]
If the limiting minimizer is unique, every sequence of exact minimizers converges to it.
\end{lemma}

\begin{proof}
Let $h_t$ minimize $\phi_t$ on $F_t$.  Compactness gives convergent subsequences.  If $h_{t_n}\to h$, Hausdorff convergence implies $h\in F$, and uniform convergence of $\phi_t$ together with continuity of $\phi$ gives
\[
\phi_{t_n}(h_{t_n})\to\phi(h)\ge\inf_F\phi.
\]
Conversely, for any $h\in F$, Hausdorff convergence gives $\widetilde h_t\in F_t$ with $\widetilde h_t\to h$, so
\[
\limsup_t\inf_{F_t}\phi_t
\le
\lim_t\phi_t(\widetilde h_t)
=
\phi(h).
\]
Taking the infimum over $h\in F$ proves convergence of the values.  Uniqueness of the limiting minimizer then forces convergence of the full minimizer sequence.
\end{proof}

\subsection{Repair limit}

Use assumptions (M1)--(M4) and write
\[
d_t=\frac{a_t-a_0}{t}\to\delta.
\]
Define
\[
q_t(h)=\frac{\Delta_t(a_0+t h)}{t^2},
\qquad
q(h)=\frac12(h-\delta)^\top G(h-\delta).
\]
The target Hessian stability and target localization in (M3)--(M4) imply that $q_t\to q$ uniformly on bounded sets and that, for each fixed $c\ge0$, the rescaled acceptable sets
\[
F_t^{\mathrm{rep}}(c)
=
\{h:q_t(h)\le c\}
\]
lie in a common compact set.  The same quadratic sandwich as above gives
\[
F_t^{\mathrm{rep}}(c)
\xrightarrow{d_H}
F^{\mathrm{rep}}(c)
:=
\{h:q(h)\le c\}.
\]

Define
\[
\phi_t(h)=\frac{I_{Q_0}(a_0+t h)}{t^2},
\qquad
\phi(h)=\frac12h^\top J_Ih.
\]
Assumption (M2) gives uniform convergence $\phi_t\to\phi$ on compact sets.  Lemma~\ref{lem:app-v2-compact-minima} therefore yields
\[
\frac{\cR_{P_t,Q_0}(t^2c)}{t^2}
\longrightarrow
\min_{h:q(h)\le c}
\frac12h^\top J_Ih.
\]
The feasible set is convex and $J_I\succ0$, so the limiting repair minimizer is unique.  The rescaled information-minimizing repaired action converges to that minimizer.  This proves the repair half of Theorem~\ref{thm:local-repair-robustness-v2}.

\subsection{Failure limit for the strict global event}

Fix $c>c_0$, where
\[
c_0=\frac12\delta^\top G\delta=q(0).
\]
The exact robustness radius uses strict target failure $\Delta_t(a)>t^2c$.  In rescaled coordinates define
\[
F_t^{>}(c)=\{h:q_t(h)>c\},
\qquad
F^{\ge}(c)=\{h:q(h)\ge c\}.
\]
The limiting optimization is over the closed set $F^{\ge}(c)$ because the infimum over the strict exterior can be approached from its boundary.

Choose any $\bar h$ with $q(\bar h)>c$.  Uniform convergence gives $q_t(\bar h)>c$ for all sufficiently small $t$, so
\[
\inf_{F_t^{>}(c)}\phi_t
\le
\phi_t(\bar h)
=O(1).
\]
Equivalently, the exact strict-failure value is $O(t^2)$.  The information-gap part of (M4) then forces every $O(t^2)$ near-minimizer into the fixed neighborhood $U$ of $a_0$.  Inside $U$,
\[
I_{Q_0}(a)\ge c_I\|a-a_0\|^2,
\]
so any such near-minimizer satisfies $\|a-a_0\|=O(t)$.  Hence the corresponding rescaled near-minimizers lie in a common compact set $K$.

For the lower bound, choose $h_t\in F_t^{>}(c)\cap K$ whose objective is within $o(1)$ of the rescaled strict-failure infimum.  Along any convergent subsequence $h_t\to h$, uniform convergence of $q_t$ gives $q(h)\ge c$, while uniform convergence of $\phi_t$ gives
\[
\liminf_t\phi_t(h_t)
\ge
\min_{q(h)\ge c}\phi(h).
\]
Therefore
\[
\liminf_t
\frac{\cM_{P_t,Q_0}(t^2c)}{t^2}
\ge
\min_{q(h)\ge c}
\frac12h^\top J_Ih.
\]

For the matching upper bound, let $h^*$ minimize $\phi$ over $q(h)\ge c$.  Since $q(0)=c_0<c$, an interior feasible point cannot minimize: if $q(h^*)>c$, the segment from $0$ to $h^*$ crosses the boundary at some $\tau h^*$ with $0<\tau<1$, and
\[
\phi(\tau h^*)=\tau^2\phi(h^*)<\phi(h^*).
\]
Thus every minimizer satisfies $q(h^*)=c$.  For $\alpha>0$, define
\[
h_\alpha
=
\delta+(1+\alpha)(h^*-\delta).
\]
Then
\[
q(h_\alpha)=(1+\alpha)^2c>c.
\]
For fixed $\alpha>0$, uniform convergence gives $q_t(h_\alpha)>c$ eventually, so
\[
\limsup_t
\frac{\cM_{P_t,Q_0}(t^2c)}{t^2}
\le
\phi(h_\alpha).
\]
Letting $\alpha\downarrow0$ yields
\[
\limsup_t
\frac{\cM_{P_t,Q_0}(t^2c)}{t^2}
\le
\phi(h^*)
=
\min_{q(h)\ge c}
\frac12h^\top J_Ih.
\]
Combining the two bounds proves
\[
\frac{\cM_{P_t,Q_0}(t^2c)}{t^2}
\longrightarrow
\min_{h:q(h)\ge c}
\frac12h^\top J_Ih.
\]
If the limiting minimizer is unique, every sequence of asymptotically minimizing strict-failure actions converges after rescaling to it.  This proves the robustness half of Theorem~\ref{thm:local-repair-robustness-v2} while preserving the strict global definition of failure.

\subsection{Generalized-mode form of the repair problem}

For $0<c<c_0$, the repair constraint is active.  The Lagrangian is
\[
\mathcal L_{\mathrm{rep}}(h,\lambda)
=
\frac12h^\top J_Ih
+\lambda\left\{
\frac12(h-\delta)^\top G(h-\delta)-c
\right\},
\qquad
\lambda\ge0.
\]
Stationarity gives
\[
(J_I+\lambda G)h
=
\lambda G\delta,
\]
so
\begin{equation}
\label{eq:app-v2-repair-KKT}
h_{\mathrm{rep}}(\lambda)
=(J_I+\lambda G)^{-1}\lambda G\delta.
\end{equation}
There is a unique $\lambda>0$ for which the regret constraint is active.  At $c=0$, the minimizer is $h=\delta$; for $c\ge c_0$, the minimizer is $h=0$.

Let $v_i$ be $J_I$-orthonormal generalized eigenvectors satisfying
\[
Gv_i=\gamma_iJ_Iv_i,
\qquad
\gamma_1\ge\cdots\ge\gamma_d>0,
\]
and write
\[
\delta=\sum_i d_iv_i,
\qquad
h=\sum_i x_iv_i.
\]
Equation~\eqref{eq:app-v2-repair-KKT} becomes
\[
x_i^{\mathrm{rep}}
=
\frac{\lambda\gamma_i}{1+\lambda\gamma_i}d_i,
\]
where $\lambda$ is the unique positive solution of
\[
\frac12\sum_i
\frac{\gamma_i d_i^2}{(1+\lambda\gamma_i)^2}
=c.
\]

\subsection{Generalized-mode form of the failure problem}

For $c>c_0$, every limiting minimizer lies on
\[
\frac12(h-\delta)^\top G(h-\delta)=c.
\]
Using the equality-constrained Lagrangian
\[
\mathcal L_{\mathrm{fail}}(h,\lambda)
=
\frac12h^\top J_Ih
-\lambda\left\{
\frac12(h-\delta)^\top G(h-\delta)-c
\right\},
\qquad
\lambda\ge0,
\]
stationarity gives
\[
(J_I-\lambda G)h
=-\lambda G\delta.
\]
On the ordinary branch, the second-order condition requires
\[
0<\lambda<\gamma_1^{-1},
\]
and the generalized coordinates satisfy
\[
x_i^{\mathrm{fail}}(\lambda)
=-\frac{\lambda\gamma_i}{1-\lambda\gamma_i}d_i.
\]
The active boundary condition is the secular equation
\[
\Phi(\lambda)
:=
\frac12\sum_i
\frac{\gamma_i d_i^2}{(1-\lambda\gamma_i)^2}
=c.
\]
The function $\Phi$ is strictly increasing on $(0,\gamma_1^{-1})$ whenever $\delta\ne0$, with $\Phi(0)=c_0$.

If $\delta$ has a nonzero component in the leading generalized eigenspace $E_1$, then
\[
\Phi(\lambda)\to\infty
\qquad
\text{as }\lambda\uparrow\gamma_1^{-1},
\]
so every $c>c_0$ has a unique ordinary-branch solution.

If instead
\[
\delta\perp_{J_I}E_1,
\]
equivalently $d_i=0$ whenever $\gamma_i=\gamma_1$, the endpoint value is finite:
\[
c_{\mathrm{hard}}
:=
\lim_{\lambda\uparrow\gamma_1^{-1}}\Phi(\lambda)
=
\frac12
\sum_{\gamma_i<\gamma_1}
\frac{\gamma_i d_i^2}{(1-\gamma_i/\gamma_1)^2}.
\]
For $c_0<c<c_{\mathrm{hard}}$, the secular equation still has a unique solution with $0<\lambda<\gamma_1^{-1}$.  At $c=c_{\mathrm{hard}}$, the multiplier reaches $\lambda=\gamma_1^{-1}$ and the leading-eigenspace component remains zero.  For $c>c_{\mathrm{hard}}$, the nonleading coordinates are fixed at
\[
x_i^0
=-\frac{\gamma_i}{\gamma_1-\gamma_i}d_i,
\qquad
\gamma_i<\gamma_1,
\]
while the leading component $z\in E_1$ satisfies
\[
\norm{z}_{J_I}^2
=
\frac{2(c-c_{\mathrm{hard}})}{\gamma_1}.
\]
Thus the minimizing set is
\[
h^0+z,
\qquad
h^0=\sum_{\gamma_i<\gamma_1}x_i^0v_i,
\]
with two choices when $E_1$ is one-dimensional and a sphere of choices when its dimension is larger.  This is the standard hard case of the corresponding trust-region problem after the change of variables $y=G^{1/2}(h-\delta)$ \cite[Chapter~4]{nocedalwright2006}.

When $\delta=0$, we have $c_0=c_{\mathrm{hard}}=0$, so every $c>0$ lies on the hard-case branch.  The Rayleigh quotient then gives
\[
m(c;0)=\frac{c}{\gamma_1},
\]
and any leading generalized eigenvector scaled to satisfy $h^\top Gh/2=c$ is a cheapest local failure direction.

\subsection{Why uniformity and localization are separate requirements}

Uniform expansions on bounded $h$-sets identify the limiting objective and boundary geometry.  Localization ensures that optimizing sequences actually remain inside such bounded sets.  Neither property implies the other.  In particular, the local expansion of $I_{Q_0}(a_0+t h)$ does not exclude a remote action whose exact contracted information is also $O(t^2)$.  The information-gap condition in (M4) supplies the global separation needed to rule out that branch.

%% file: appendix-restricted-families.tex
\section{Restricted-Family Operator Details and Gaussian Checks}
\label{app:restricted-family-details}

This appendix supplies the operator calculations behind Section~\ref{sec:restricted-families} and the finite Gaussian checks used in Section~\ref{sec:examples}.  The first part separates restricted decision coverage from restricted Fisher cost at the operator level.  The second part works through Gaussian examples where the two effects can be seen directly.

\subsection{Minimum energy in a restricted tangent space}

Let $\mathcal H=L_0^2(Q_0)$, let $E=\operatorname{ran}B$ be the finite-dimensional effective action space, and let $T\subset\mathcal H$ be the closed linear tangent space used to approximate the family locally.  Write
\[
B_T=B|_T:T\to E.
\]
Its adjoint is
\[
B_T^*=\Pi_TB^*,
\]
so
\[
B_TB_T^*
=
B\Pi_TB^*
=:C_T.
\]
Because $E$ is finite-dimensional,
\[
\operatorname{ran}(C_T)
=
\operatorname{ran}(B_T)
=B(T)
=:S_T.
\]
Indeed,
\[
\ker(C_T)=\ker(B_T^*)=\operatorname{ran}(B_T)^\perp,
\]
and taking orthogonal complements gives the result.

For $h\in S_T$, the minimum-norm solution of $B_Ts=h$ is
\[
s_{T,h}^*
=B_T^*(B_TB_T^*)^\dagger h
=\Pi_TB^*C_T^\dagger h.
\]
Therefore
\begin{equation}
\label{eq:app-v2-restricted-energy}
\inf_{s\in T:B[s]=h}
\frac12\norm{s}_{L^2(Q_0)}^2
=
\frac12h^\top C_T^\dagger h,
\qquad h\in S_T.
\end{equation}
If $h\notin S_T$, the constrained set is empty and the value is $+\infty$.  This proves Proposition~\ref{prop:restricted-local-information-v2}.

Equation~\eqref{eq:app-v2-restricted-energy} is a statement about the linearized family.  For a nonlinear approximation family $\cQ$, identifying this quadratic form with the Hessian of the exact restricted contraction would require
\[
I_{Q_0,\cQ}(a_0+t h)
=
\frac{t^2}{2}h^\top C_T^\dagger h+o(t^2),
\]
which in turn requires regular family paths realizing the relevant tangent scores, a compatible second-order KL expansion, and localization.  The main text therefore uses $C_T^\dagger$ as a tangent Fisher-energy diagnostic without assuming this stronger nonlinear identification.

\subsection{Restricted versus ambient tangent cost}

The ambient minimum-energy lift of a reachable displacement $h$ is
\[
s_h^*=B^*J_Ih,
\qquad
J_I=(BB^*)^{-1},
\]
with tangent cost
\[
I^{\mathrm{tan}}(h)
=
\frac12h^\top J_Ih.
\]
The restricted tangent cost is
\[
I_T^{\mathrm{tan}}(h)
=
\frac12h^\top C_T^\dagger h.
\]
Since the restricted problem minimizes over the smaller score space $T$,
\[
I_T^{\mathrm{tan}}(h)
\ge
I^{\mathrm{tan}}(h),
\qquad h\in S_T.
\]
For $S_T\neq\{0\}$, define
\[
\kappa_T
=
\sup_{h\in S_T\setminus\{0\}}
\frac{h^\top C_T^\dagger h}
{h^\top J_Ih}.
\]
When $S_T=\{0\}$, set $\kappa_T=1$ by convention.

\begin{proposition}[Equality condition for local information efficiency]
\label{prop:app-v2-kappa-equality}
The restricted and ambient tangent costs agree for every $h\in S_T$ if and only if
\[
B^*J_I(S_T)\subseteq T.
\]
With the convention $\kappa_T=1$ when $S_T=\{0\}$, this is equivalent to $\kappa_T=1$ exactly when the restricted tangent contains the ambient minimum-information lift of every decision displacement it can realize.
\end{proposition}

\begin{proof}
For each reachable $h$, the ambient minimum-norm solution $s_h^*=B^*J_Ih$ is unique because it is the orthogonal projection of the origin onto the affine space
\[
\{s\in\mathcal H:B[s]=h\}.
\]
If $s_h^*\in T$, it is feasible for the restricted problem and the two minima agree.  Conversely, if the minima agree, a restricted minimizer must attain the unique ambient minimum norm, hence must equal $s_h^*$ and therefore lie in $T$.
\end{proof}

Thus information efficiency on a proper reachable image requires only the horizontal lifts of directions in $S_T$ to lie in $T$; the entire range of $B^*$ need not be contained in the restricted tangent.

\subsection{Coverage residual and the scalar defect \texorpdfstring{$\eta_T$}{etaT}}

If $S_T\neq\{0\}$, let $U$ be a full-column-rank matrix whose columns span $S_T$.  The $G$-orthogonal projection onto $S_T$ is
\[
P_T^G
=
U(U^\top GU)^{-1}U^\top G,
\]
and the corresponding residual matrix is
\[
R_T
=
G-GU(U^\top GU)^{-1}U^\top G.
\]
If $S_T=\{0\}$, set $R_T=G$.  In either case,
\[
\inf_{s\in S_T}
\frac12(h-s)^\top G(h-s)
=
\frac12h^\top R_Th.
\]
Whitening by the ambient information metric gives
\[
\mathcal E_\rho(T)
=
\sup_{\|z\|^2\le2\rho}
\frac12z^\top
J_I^{-1/2}R_TJ_I^{-1/2}z.
\]
Hence
\[
\mathcal E_\rho(T)=\rho\eta_T,
\qquad
\eta_T
=
\lambda_{\max}(J_I^{-1/2}R_TJ_I^{-1/2}),
\]
which proves Proposition~\ref{prop:restricted-family-coverage-v2}.  If $r_T=\dim S_T$, Theorem~\ref{thm:effective-dimension} gives
\[
\eta_T\ge\gamma_{r_T+1},
\]
with equality when $S_T$ is an oracle rank-$r_T$ generalized-eigenvector subspace.

\subsection{Equal rank can have different coverage}

Consider
\[
J_I=
\begin{pmatrix}
1&0\\0&4
\end{pmatrix},
\qquad
G=
\begin{pmatrix}
5&6\\6&20
\end{pmatrix}.
\]
Information whitening gives
\[
J_I^{-1/2}GJ_I^{-1/2}
=
\begin{pmatrix}
5&3\\3&5
\end{pmatrix},
\]
with generalized eigenvalues $8$ and $2$.  The leading $J_I$-orthonormal mode is
\[
v_1=2^{-1/2}(1,1/2)^\top,
\]
so the oracle rank-one image $S_*=\operatorname{span}(v_1)$ has
\[
\eta_*=2.
\]

For $S_1=\operatorname{span}(e_1)$, optimizing the retained first coordinate leaves residual regret
\[
\frac12\left(20-\frac{36}{5}\right)y^2
=
\frac{32}{5}y^2.
\]
Under
\[
\frac12(x^2+4y^2)\le\rho,
\]
we have $y^2\le\rho/2$, hence
\[
\mathcal E_\rho(S_1)=\frac{16}{5}\rho.
\]
For $S_2=\operatorname{span}(e_2)$, the residual is
\[
\frac12\left(5-\frac{36}{20}\right)x^2
=
\frac{8}{5}x^2,
\]
and $x^2\le2\rho$, giving the same value.  Therefore
\[
\eta_{S_*}=2,
\qquad
\eta_{S_1}=\eta_{S_2}=\frac{16}{5}.
\]
All three images have rank one, but only the generalized-eigenvector image attains the oracle coverage bound.

\subsection{Centered mean-field Gaussian tangent geometry}

Let $Q_0=N(0,I_2)$ and compare the centered full-Gaussian family
\[
\cQ_{\mathrm{full}}
=\{N(0,\Sigma):\Sigma\succ0\}
\]
with the centered mean-field family
\[
\cQ_{\mathrm{MF}}
=\{N(0,\operatorname{diag}(\sigma_1^2,\sigma_2^2)):\sigma_1^2,\sigma_2^2>0\}.
\]
For a symmetric matrix $H$, the covariance path
\[
Q_t=N(0,I_2+tH)
\]
has score
\begin{equation}
\label{eq:app-v2-gaussian-score}
s_H(\theta)
=
\frac12\left\{
\theta^\top H\theta-\operatorname{tr}(H)
\right\}.
\end{equation}
Standard Gaussian fourth-moment identities give
\begin{equation}
\label{eq:app-v2-gaussian-score-norm}
\norm{s_H}_{L^2(Q_0)}^2
=
\frac12\operatorname{tr}(H^2).
\end{equation}
The full centered Gaussian tangent contains both diagonal and off-diagonal quadratic scores.  The centered mean-field tangent contains the variance scores $\theta_1^2-1$ and $\theta_2^2-1$ but not the correlation score $\theta_1\theta_2$.

\subsection{Second-moment decision: pure coverage loss}

Let
\[
\phi(\theta)
=
\begin{pmatrix}
\theta_1^2\\
\sqrt2\,\theta_1\theta_2\\
\theta_2^2
\end{pmatrix}.
\]
Under weighted squared loss for the action $a\in\mathbb R^3$, a covariance perturbation $H$ induces action displacement
\[
h=(H_{11},\sqrt2H_{12},H_{22})^\top.
\]
Because
\[
\operatorname{tr}(H^2)=\norm h^2,
\]
Equations~\eqref{eq:app-v2-gaussian-score}--\eqref{eq:app-v2-gaussian-score-norm} give
\[
\frac12h^\top J_Ih
=
\frac14\norm h^2,
\qquad
J_I=\frac12I_3.
\]
With
\[
W=\operatorname{diag}(1,\omega,1),
\]
the regret Hessian is $G=2W$.  The centered mean-field decision image is
\[
S_{\mathrm{MF}}=\operatorname{span}\{e_1,e_3\},
\]
so the residual contains only the cross-moment coordinate and
\[
\eta_{\mathrm{MF}}=4\omega.
\]
The ambient minimum-information lifts of the two reachable diagonal directions already lie in the mean-field tangent, hence
\[
\kappa_{\mathrm{MF}}=1.
\]
This is the pure coverage defect used in Section~\ref{subsec:mf-covariance-example-v2}.

\subsection{Mixed scalar moment: complete coverage with excess cost}

Let
\[
g_\alpha(\theta)
=
\theta_1\theta_2+\alpha(\theta_1^2-1),
\qquad
\alpha\ne0.
\]
At $Q_0=N(0,I_2)$,
\[
\operatorname{Var}(\theta_1\theta_2)=1,
\qquad
\operatorname{Var}(\theta_1^2-1)=2,
\]
and the two Hermite components are orthogonal.  Thus
\[
BB^*
=
\operatorname{Var}(g_\alpha)
=
1+2\alpha^2,
\qquad
J_I=(1+2\alpha^2)^{-1}.
\]
Projection onto the centered mean-field tangent removes the correlation term, leaving
\[
\Pi_{T_{\mathrm{MF}}}g_\alpha
=
\alpha(\theta_1^2-1),
\]
so
\[
C_{\mathrm{MF}}
=B\Pi_{T_{\mathrm{MF}}}B^*
=2\alpha^2.
\]
Because $C_{\mathrm{MF}}>0$, the scalar decision image is all of $\mathbb R$ and
\[
\eta_{\mathrm{MF}}=0.
\]
The information-inflation factor is
\[
\kappa_{\mathrm{MF}}
=
\frac{1+2\alpha^2}{2\alpha^2}
=
1+\frac1{2\alpha^2}.
\]
At $\alpha=1/\sqrt2$, this gives $\kappa_{\mathrm{MF}}=2$.  This is the complete-coverage, inefficient-lift example used in Section~\ref{subsec:mf-efficiency-example-v2}.

\subsection{Exact finite-Gaussian check of the local coverage approximation}

Within the centered full-Gaussian family, fix an off-diagonal covariance $c$ and minimize
\[
\KL\!\left(N(0,\Sigma)\,\middle\|\,N(0,I_2)\right)
\]
over positive-definite covariances with $\Sigma_{12}=c$.  Symmetry and the first-order condition give
\[
\Sigma_c
=
\begin{pmatrix}
x&c\\c&x
\end{pmatrix},
\qquad
x=\frac{1+\sqrt{1+4c^2}}2,
\]
and the defining equation for $x$ gives
\[
\det(\Sigma_c)=x.
\]
The exact Gaussian information radius is therefore
\[
\rho(c)
=x-1-\frac12\log x.
\]
The centered mean-field family has zero cross second moment for every member.  After matching the reachable diagonal coordinates, its residual weighted regret is
\[
\Delta_{\mathrm{MF}}(c)=2\omega c^2.
\]
Since
\[
\rho(c)=\frac12c^2+O(c^4),
\]
we obtain
\[
4\omega\rho(c)
=
2\omega c^2+O(c^4).
\]
Thus the exact finite-Gaussian calculation is tangent to the local prediction
\[
\mathcal E_\rho(T_{\mathrm{MF}})=4\omega\rho
\]
at the baseline.

\subsection{Structured-Gaussian KL and decision-ranking check}
\label{app:structured-gaussian-check}

The final example in Section~\ref{subsec:structured-gaussian-comparison} can also be checked analytically block by block.  For a two-dimensional centered Gaussian target with unit variances and correlation $r$,
\[
P_r
=
N\!\left(
0,
\begin{pmatrix}
1&r\\r&1
\end{pmatrix}
\right),
\]
consider a centered Gaussian approximation with the same unit variances and correlation $q$.  Its reverse KL is
\[
K(q;r)
=
\frac12\left\{
\operatorname{tr}(\Sigma_r^{-1}\Sigma_q)-2
+\log\frac{1-r^2}{1-q^2}
\right\}.
\]
Differentiation gives
\[
\frac{\partial K}{\partial q}
=
-\frac{r}{1-r^2}
+
\frac{q}{1-q^2}.
\]
The map $u\mapsto u/(1-u^2)$ is strictly increasing on $(-1,1)$, so the unique minimizer is $q=r$.  If a family forbids the correlation and forces $q=0$, its KL penalty for that block is
\[
k(r)
:=
K(0;r)
=
\frac{1}{1-r^2}-1
+\frac12\log(1-r^2).
\]
For the two target correlations in Section~\ref{subsec:structured-gaussian-comparison},
\[
k(0.35)=0.0743,
\qquad
k(0.70)=0.6241.
\]
Because the four-dimensional target is block diagonal, the KL divergences add across the two blocks.  Hence
\[
\KL(Q_{12}\|P)=k(0.35)=0.0743,
\]
\[
\KL(Q_{34}\|P)=k(0.70)=0.6241,
\]
and
\[
\KL(Q_{\mathrm{MF}}\|P)=k(0.35)+k(0.70)=0.6984.
\]

For
\[
z=2^{-1/2}(1,1)^\top,
\]
a unit-variance Gaussian block with correlation $r$ satisfies
\[
\E(z^\top\theta)^2
=z^\top\Sigma_r z
=1+r.
\]
Under squared-error loss for this second moment, replacing the target correlation $r$ by zero therefore incurs regret
\[
r^2.
\]
Thus the two decision penalties in the main-text table are
\[
0.70^2=0.4900,
\qquad
0.35^2=0.1225.
\]
The same finite calculation therefore produces both orderings emphasized in the example: reverse KL favors the family that captures the stronger correlation, while each downstream second-moment decision favors the family that represents its own block.